\documentclass[a4paper]{amsart}
\usepackage{amsmath}
\usepackage{graphicx}

\usepackage{tikz}
\usetikzlibrary{arrows, automata, shapes, positioning}
\usetikzlibrary{arrows.meta}
\usetikzlibrary{decorations.pathmorphing,patterns,decorations.pathreplacing}
\usetikzlibrary{calc}

\DeclareMathOperator{\rep}{rep}
\DeclareMathOperator{\val}{val}
\newtheorem{theorem}{Theorem}
\newtheorem{corollary}[theorem]{Corollary}
\newtheorem{proposition}[theorem]{Proposition}
\newtheorem{lemma}[theorem]{Lemma}
\theoremstyle{definition}
\newtheorem{definition}[theorem]{Definition}
\newtheorem{example}[theorem]{Example}
\newtheorem{remark}[theorem]{Remark}

\title{On digital sequences associated with Pascal's triangle}
\author{P. Mathonet}
\author{M. Rigo}
\author{M. Stipulanti}
\author{N. Zena\"idi}

\begin{document}

\begin{abstract}
We consider the sequence of integers whose $n$th term has base-$p$ expansion given by the $n$th row of Pascal's triangle modulo $p$ (where $p$ is a prime number). We first present and generalize  well-known relations concerning this sequence. Then, with the great help of Sloane's On-Line Encyclopedia of Integer Sequences, we show that it appears naturally as a subsequence of a $2$-regular sequence. Its study provides interesting relations and surprisingly involves odious and evil numbers, Nim-sum and even Gray codes.
Furthermore, we examine similar sequences emerging from prime numbers involving alternating sum-of-digits modulo~$p$.
This note ends with a discussion about Pascal's pyramid involving trinomial coefficients. 
\end{abstract}
\maketitle
\section{Introduction}
The problem of determining the number of sides of constructible regular polygons (with straightedge and compass) has captivated geometers for centuries. The Gauss--Wantzel theorem \cite{Cox} translates this question in the framework of number theory: it states that a regular $n$-sided polygon is constructible if and only if $n$ is the product of a power of $2$ and any number (possibly none) of distinct Fermat primes i.e., primes of the form $F_l:=2^{2^l}+1$, $l\ge 0$. 

The sequence of (ordered) products of Fermat numbers\footnote{The sequence starts with the product of no such numbers which is $1$ by convention.} $(\mathbf{f}_{2,n})_{n\ge 0}$, starting with
\[1, \underline{3}, \underline{5}, 15, \underline{17}, 51, 85, 255, \underline{257}, 771,1285, 3855, 4369, 13107, 21845, 65535, \underline{65537},\ldots\]
actually  appears as entry {\tt A001317} in Sloane's Encyclopedia \cite{Sloane}.
Note that the sequence of Fermat primes, which appears in the OEIS \cite{Sloane} under entry {\tt A019434}, is not completely known: the primality of Fermat numbers is in general an open problem and only the numbers $F_0,F_1,\ldots,F_4$ (that have been underlined here) are known to be prime.

This sequence has in turn intriguing properties from the point of view of cominatorics. It turns out that it can be extracted from Pascal's triangle. Indeed, for any prime number $p$, we can consider the elements of Pascal's triangle modulo $p$, ${n \choose i} \bmod{p}$, where for any integer $k$, we let $k\bmod p$ or even $[k\% p]$ denote the unique integer in $\{0,\ldots,p-1\}$ congruent to $k$ modulo $p$. For instance, for $p=2$,  we get
\[
\left[ {n \choose i} \bmod{p}\right]_{\substack{n\ge 0, \\ 0\le i \le n}} = 
1 | 1, 1 | 1, 0, 1 | 1, 1, 1, 1 | 1, 0, 0, 0, 1 | 1, 1, 0, 0, 1, 1 |\cdots
\]
where bars separate consecutive terms in the sequence. We thus find Sierpi\'nski's triangle (the usual Pascal's triangle modulo 2). 

Identifying each row of the triangle to an integer through the base $p$-expansion, we get the sequence $(\mathbf{t}_{p,n})_{n\ge 0}$ defined, for $n\ge 0$, by 
\[\mathbf{t}_{p,n}=\sum_{i=0}^n \left[ {n \choose i} \bmod{p}\right] p^i.\]
These numbers $\mathbf{t}_{2,n}$ are sometimes called {\em Roberts' numbers} \cite{Granville}. It is easily seen from the examples above that the first few terms of the sequences $\mathbf{f}_{2,n}$ and $\mathbf{t}_{2,n}$ coincide. It was observed by several authors that the sequences are indeed equal. Conway and Guy refer to Gardner \cite{Conway,Gardner}. Investigating connections of Fermat numbers with Pascal's triangle, Krizek {\em et al.} in their monograph \cite[Chap.~8]{Krizek} mention the earlier work of Hewgill \cite{Hewgill}. 

For $p=3$, the sequence $\mathbf{t}_{3,n}$ also appears in the OEIS \cite{Sloane} as entry {\tt A173019}
$$(\mathbf{t}_{3,n})_{n\ge 0}=1, 4, 16, 28, 112, 448, 784, 3136, 12301, 19684, 78736, 314944,\ldots.$$
It turns out that it no longer coincides with the natural generalization $\mathbf{f}_{3,n}$ of $\mathbf{f}_{2,n}$ defined as the (ordered) products of numbers of the form $3^{3^l}+1,\,l\ge 0$.
 
In Section~\ref{sec:rec}, considering the $n$th row of Pascal's triangle mod $p$, we write $n=n_k\,p^k+s$ where $p^k$ is the largest power of $p$ smaller or equal to $n$, $n_k\in\{1,\ldots,p-1\}$ and $s=n\bmod p^k$. We give a simple formula that expresses $\mathbf{t}_{p,n}$ from  $\mathbf{t}_{p,s}$. This formula enables us to recover the equality of the sequences $\mathbf{f}_{2,n}$ and $\mathbf{t}_{2,n}$ and explain why in general $\mathbf{f}_{p,n}\neq\mathbf{t}_{p,n}$. We also interpret this formula in terms of self-similarity of Pascal's triangle mod $p$.

In Section~\ref{sec: polynomials}, we consider another approach: we study polynomial identities whose evaluation at the specific value $p$ gives the sequence $(\mathbf{t}_{p,n})_{n\ge 0}$. Evaluations at other values give generalized versions of these sequences having similar properties.

In Section~\ref{sec:nim}, we start with the simple observation\footnote{It follows directly from Pascal's rule.} that $\mathbf{t}_{2,n+1}=\mathbf{t}_{2,n}\oplus (2\, \mathbf{t}_{2,n})$ where $\oplus$ is the classical Nim-sum (addition digit-wise modulo~$2$ without carry). This leads us to study another related sequence $(N(m))_{m\ge 0}:=(m\oplus 2m)_{m\ge 0}$ of which $(\mathbf{t}_{2,n})_{n\ge 0}$ is a subsequence, and more generally, the sequence $(N_p(m))_{m\ge 0}:=(m\oplus_p p\, m)_{m\ge 0}$, where $\oplus_p$ is addition digit-wise modulo~$p$ without carry. We show that $N_p(m)$ is a $p$-regular sequence whereas $\mathbf{t}_{p,n}$ is not.

In Section~\ref{sec:par}, we consider a particular partition of the set $\{N(m)\mid m\ge 1\}$. It turns out that $(N(m))_{m\ge 0}$ is a well-understood permutation of the sequence of evil numbers, those numbers whose base-$2$ expansions have an even number of ones (i.e., the characteristic sequence of the considered set is given by the Thue--Morse sequence). We then naturally extend this result to any prime $p$ by showing that the characteristic sequence of the set $\{N_p(m)\mid m\ge 0\}$ is a generalization of the Thue--Morse sequence: it is the set of numbers whose alternate sum-of-digit is zero modulo~$p$.
Finally, in \cite{All}, an exact formula for the summatory function of evil numbers is given. Here we consider the summatory function of $(N(m))_{m\ge 0}$ taking advantage of the known permutation.

In Section~\ref{sec:trinomial}, we examine the problem in three dimensions and consider Pascal's pyramid made of trinomial coefficients. We define a sequence analogous to $\mathbf{t}_{p,n}$: when the pyramid is intersected with convenient planes whose equation is of the form $x+y+z=n$ for some integer $n$, we get rows of coefficients modulo~$p$. Similarly to what is done in Section~\ref{sec:rec}, we derive a recurrence relation for the corresponding integer sequence.  Finally we study the relation existing between coefficients modulo $p$ occurring at specific positions. In particular, we show that Pascal's pyramid modulo $p$ is $p$-automatic.

Note that, in this text, appear the sequences {\tt A001317}, {\tt A001969}, {\tt A003188}, {\tt A019434}, {\tt A048724}, {\tt A071770}, {\tt A173019} and {\tt A242399} from the OEIS \cite{Sloane}.

\section{A recursive formula}\label{sec:rec}

In this section, for any prime $p$, considering an index $n\neq 0$ whose base $p$-expansion is $n=\sum_{\ell=0}^kn_\ell p^\ell$, with $n_k\neq 0$, we let $\rep_p(n)$ denote the \textit{base-$p$ representation} of the integer $n>0$ i.e., the word\footnote{By convention, $0$ is represented by the empty word $\varepsilon$.} $n_k\cdots n_0$ over $\{0,\ldots,p-1\}$. We also decompose $n$ as $n=n_k\,p^k+s$ where $0<n_k<p$ and $s= n \mod p^k$. With Theorem~\ref{thm:sec1} we describe a formula to compute $\mathbf{t}_{p,n}$ from $\mathbf{t}_{p,s}$. We essentially follow the same description as in \cite{Granville}, which makes use of Lucas' theorem, which we now recall. 
\begin{theorem}[Lucas]Let $p\ge 2$ be a prime and let $m,n$ be non-negative integers. If $\rep_p(m)=m_k\cdots m_0$ and $\rep_p(n)=n_k\cdots n_0$\footnote{With the convention that the shortest representation is padded with extra leading zeroes as most significant digits when the two base-$p$ representations have different lengths.}, then
$$\binom{m}{n}\equiv \prod_{j=0}^k \binom{m_j}{n_j}\pmod{p},$$
with the convention that $\binom{a}{b}=0$ if $a<b$.
\end{theorem}
In our situation, we use this theorem to compute ${n\choose i}$ for $n=\sum_{\ell=0}^kn_\ell p^\ell$, with $n_k\neq 0$ and $i\leq n$. We write $i=\sum_{\ell=0}^ki_\ell p^\ell$ (allowing leading zeroes whenever $i<p^k$) and since $i\le n$, we have $0\le i_k\le n_k$. From Lucas' theorem, we then have
\begin{equation}\label{eq:Lucas-mult-mu}
\binom{n}{i}\equiv\binom{n_k}{i_k}\binom{n\bmod p^k}{i\bmod p^k} \pmod{p}.
\end{equation}

\begin{proposition}\label{prop:sec1}
If $n=\sum_{\ell=0}^{k} n_\ell p^\ell$ with $n_k \neq 0$ and if $s=n\bmod p^k$, then we have 
\begin{equation}
\label{eq: recurrence}
\mathbf{t}_{p,n} = \sum_{m=0}^{n_k} \left[\sum_{j=0}^{s} \left({n_k \choose m} {s \choose j} \bmod p \right) p^j\right] p^{mp^k}.
\end{equation}
\end{proposition}
\begin{proof}
By definition and Equation \eqref{eq:Lucas-mult-mu} (and using the same notation) we have 
\[\mathbf{t}_{p,n} =\sum_{i=0}^n \left[ {n \choose i} \bmod{p}\right] p^i=\sum_{i=0}^n \left[\binom{n_k}{i_k}\binom{n\bmod p^k}{i\bmod p^k} \bmod{p}\right] p^i.\] 
In order to conclude, it is sufficient to observe that the set of indices $i\leq n$ satisfying the condition $\binom{n_k}{i_k}\binom{n\bmod p^k}{i\bmod p^k}\neq 0\bmod{p}$ is precisely the set of those indices $i$ which satisfy the same condition and decompose as $i=i_k\,p^k+(i\bmod p^k)$, with $i_k\le n_k$, $i\bmod p^k\le s$.
We thus have 
\[\mathbf{t}_{p,n} =\sum_{m=0}^{n_k} \left[\sum_{j=0}^s\binom{n_k}{m}\binom{s}{j} \bmod{p}\right] p^{mp^k+j},\]
and the result follows.
\end{proof}
We now interpret Proposition \ref{prop:sec1} as a recursive formula relating $\mathbf{t}_{p,n}$ and $\mathbf{t}_{p,s}$. We simply observe that the expression within brackets in \eqref{eq: recurrence} is obtained by multiplying each digit of $\rep_p(\mathbf{t}_{p,s})$ by ${n_k \choose m}$ and then taking the remainder mod $p$. In order to express this fact, we introduce some notation.

We denote by $\mathbb{Z}/(p\mathbb{Z})$ or simply $\mathbb{Z}_p$ the field of integers modulo $p$. The canonical projection $\pi\colon\mathbb{Z}\to\mathbb{Z}_p$ induces a bijection from $\mathbb{Z}_{<p}=\{0,\ldots,p-1\}$ to $\mathbb{Z}_p$ and enables the definitions of operations on $\mathbb{Z}_{<p}$ for which $\pi$ is an isomorphism. These are simply the addition and multiplication mod $p$.
For integers $0\le b \le a$ such that $\binom{a}{b}\neq 0\pmod p$ we let $\mu_{a,b}$ denote the one-to-one correspondence 
$$\mu_{a,b}:\mathbb{Z}_{<p}\to\mathbb{Z}_{<p}\colon x\mapsto \binom{a}{b} \cdot x\bmod{p}.$$
This map is a permutation of $\mathbb{Z}_{<p}$ since it corresponds via $\pi$ to the left multiplication by the class of $\binom{a}{b}\neq 0$ in $\mathbb{Z}_p$. Also, we have $\mu_{a,b}(x)=0$ if and only if $x=0$. This fact will be extensively used in this text (notably in Section~\ref{sec:trinomial}).
Note also that for $b=0$ or $a=b$, $\mu_{a,b}$ is just the identity. 
We extend this map to a morphism of the free monoid $\{0,\ldots,p-1\}^*$ equipped with concatenation by setting
$$\mu_{a,b}(z_m\cdots z_0)=\mu_{a,b}(z_m)\cdots \mu_{a,b}(z_0).$$ 
Finally, for a finite sequence $\delta_k\cdots \delta_0$ of digits in $\{0,\ldots,p-1\}$, we let $\val_p(\delta_k\cdots \delta_0)$ denote the \textit{$p$-evaluation} $\sum_{i=0}^k \delta_i\, p^i$. 
We are now able to translate Proposition \ref{prop:sec1}.
\begin{theorem}\label{thm:sec1}
If $n=\sum_{\ell=0}^{k} n_\ell p^\ell$ with $n_k \neq 0$ and if $s=n\bmod p^k$, then we have 
\begin{equation}
\label{eq: recurrence1}
\mathbf{t}_{p,n} = \sum_{m=0}^{n_k} p^{mp^k} \val_p(\mu_{n_k,m}(\rep_p(\mathbf{t}_{p,s}))).
\end{equation}
\end{theorem}
\begin{proof}
By definition, we have $\rep_p(\mathbf{t}_{p,s})=(\binom{s}{j}\bmod p)_{j=0,\ldots,s}$. Then we compute 
\[\mu_{n_k,m}(\rep_p(\mathbf{t}_{p,s}))=(\binom{n_k}{m}\,\binom{s}{j}\bmod p)_{j=0,\ldots,s}\]
and finally
\[\val_p(\mu_{n_k,m}(\rep_p(\mathbf{t}_{p,s})))=\sum_{j=0}^s(\binom{n_k}{m}\,\binom{s}{j}\bmod p)p^j,\]
and the result follows from Proposition \ref{prop:sec1}.
\end{proof}

Theorem \ref{thm:sec1} now has an interpretation in terms of the rows of Pascal's triangle mod $p$, as follows. We start indexing rows and columns of Pascal's triangle at $0$. The $n$th row of this triangle is nothing but $\rep_p(\mathbf{t}_{p,n})$ and the $s$th is the word $P_s:=\rep_p(\mathbf{t}_{p,s}):=t_{s,0}\cdots t_{s,s}$ made of $s+1$ entries over $\mathbb{Z}_{<p}$. Equation \eqref{eq: recurrence1} suggests to work with words of length $p^k$, so we pad this row with trailing zeroes:
$$Q_s:=t_{s,0}\cdots t_{s,s} \underbrace{0\cdots 0}_{p^k-s-1}.$$

\begin{example} Let $p=5$ and $n=23$. We have $\rep_5(23)=43$ thus $k=1$, $n_1=4$ and $s=3$.
The third row of Pascal's triangle (modulo $5$) is with $1,3,3,1$, so $P_3=1331$ and $Q_3=13310$ (see Figure~\ref{fig:04} where different colors represent different values modulo $5$).
\begin{figure}[b]
  \centering
  \includegraphics[height=2cm]{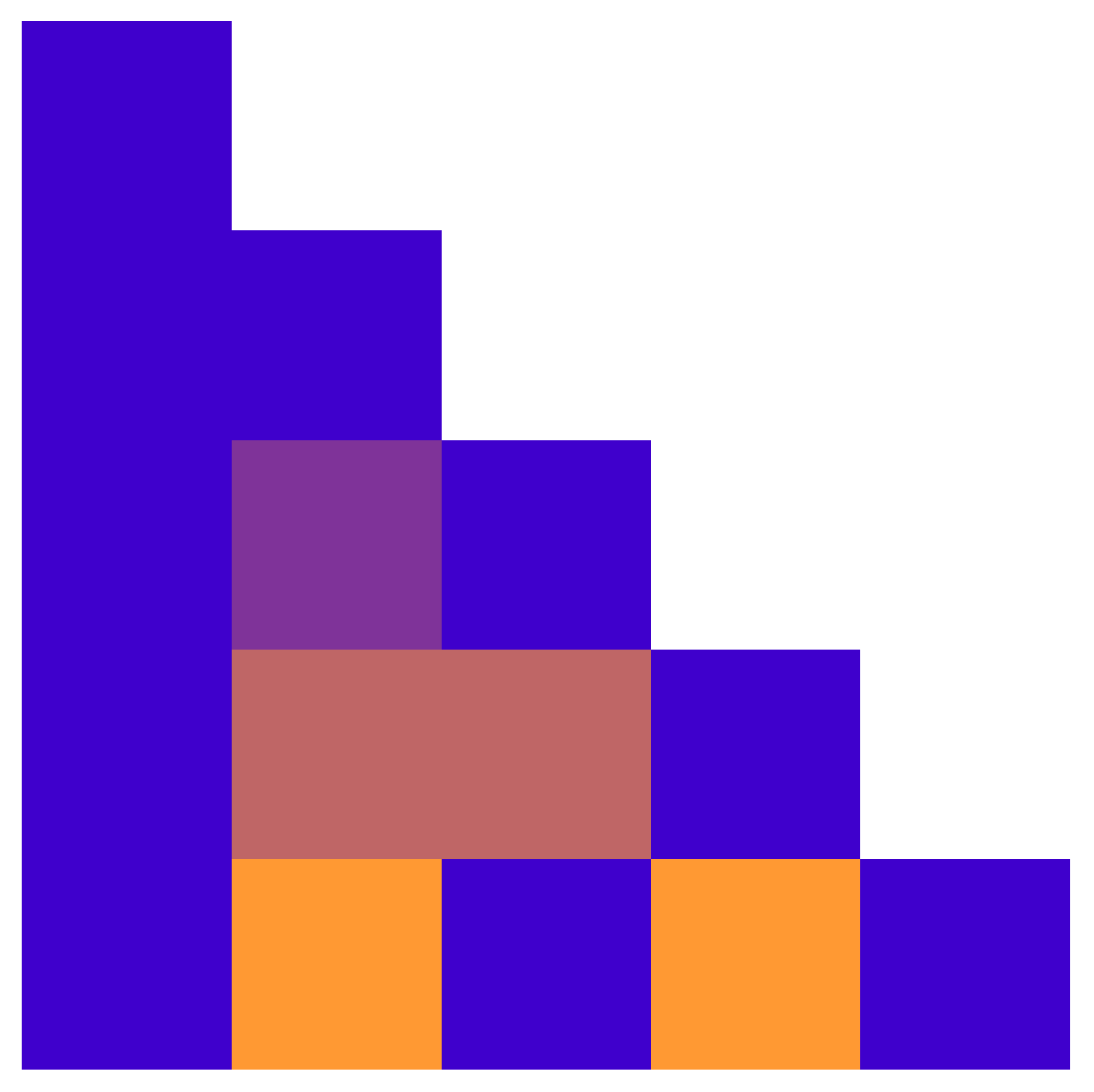}
  \caption{The first five rows of Pascal's triangle mod~$5$.}
  \label{fig:04}
\end{figure}
\end{example}
From Theorem \ref{thm:sec1} we directly obtain
$$P_n=\mu_{n_k,0}(Q_s)\, \mu_{n_k,1}(Q_s)\, \cdots\, \mu_{n_k,n_k-1}(Q_s)\, \mu_{n_k,n_k}(P_s).$$
By definition of $Q_s$, we have
$$P_n=\mu_{n_k,0}(P_s) 0^{p^k-s-1} \mu_{n_k,1}(P_s) 0^{p^k-s-1} \cdots\, \mu_{n_k,n_k-1}(P_s) 0^{p^k-s-1} \mu_{n_k,n_k}(P_s).$$
For instance, the third row will help describe the $23$rd row, for which 
\[P_{23}=133104224013310422401331\] (see Figure~\ref{fig:2024}).
\begin{figure}[h!tbp]
  \centering
  \includegraphics[height=2cm]{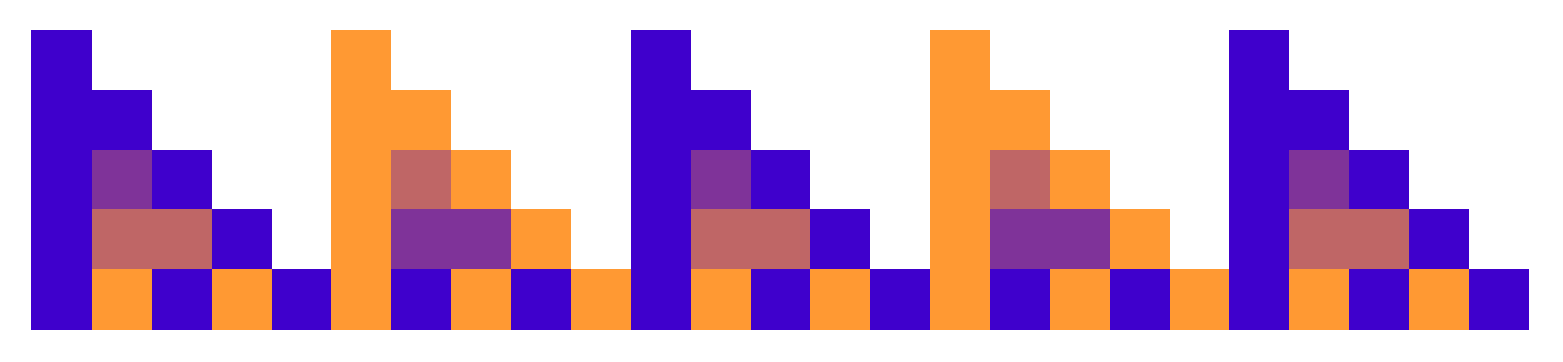}
  \caption{The rows $20\le i\le 24$ of Pascal's triangle mod~$5$.}
  \label{fig:2024}
\end{figure}
Similarly, for the row with index $48$, we have $\rep_5(48)=143$ thus $k=2$, $n_k=1$ and $s=23$. So $P_{48}$ is described in terms of $P_{23}$ and therefore in terms of $P_3$ and $Q_3$ (see Figure~\ref{fig:5054}).
\begin{figure}[h!tbp]
  \centering
  \includegraphics[height=1.8cm]{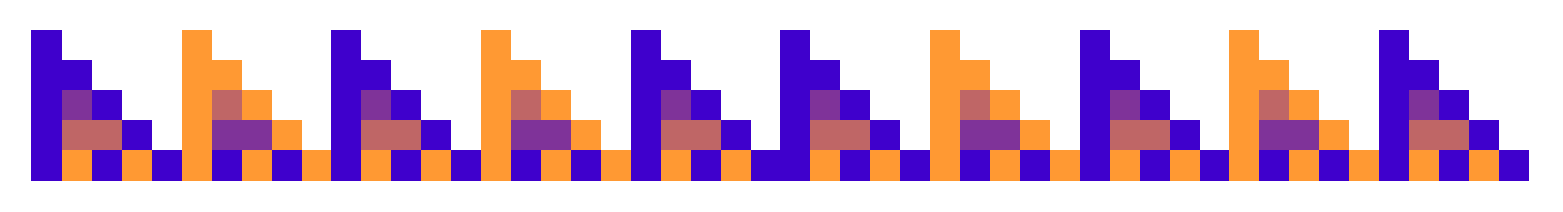}
  \caption{The rows $45\le i\le 49$ of Pascal's triangle mod~$5$.}
  \label{fig:5054}
\end{figure}

For the special case $p=2$, we have $n_k=1$, no permutation is needed in Theorem \ref{thm:sec1} because $\mu_{1,0}=\mu_{1,1}$ is the identity. Also, in Equation \eqref{eq: recurrence}, with $p=2$ and $n_k=1$, the sum is restricted to two terms giving the $k$th Fermat number times $\mathbf{t}_{2,s}$. So from Theorem \ref{thm:sec1} or Proposition \ref{prop:sec1}, we recover the following result.
\begin{corollary} For all $k\ge 0$ and all $0\le s<2^k$, we have
  $$\mathbf{t}_{2,2^k+s}=(2^{2^k}+1)\, \mathbf{t}_{2,s}.$$  
\end{corollary}

From this corollary, one deduces that \cite[p.~113]{AS} (an unpublished result attributed to Larry Roberts) for $n\ge 1$, 
\begin{equation}\label{eq:LarryRoberts}
\mathbf{t}_{2,n}= \prod_{j:n_j=1} (2^{2^j}+1),
\end{equation}
which is the product of the Fermat numbers for those indices that occur in the base-$2$ expansion of $n$. In particular, since the base-$2$ expansions of $2n$ and $2n+1$ only differ by their last digit, we have
\begin{equation}\label{F3}
  \frac{\mathbf{t}_{2,2n+1}}{\mathbf{t}_{2,2n}}=2^{2^0}+1=3 
\end{equation}
which is the first Fermat number $F_0$.

\subsection{A variant sequence} If one only looks at the zero or non-zero binomial coefficients modulo $p$ (so, considering only divisibility by $p$ as, for instance, in \cite{Haesler}), we can study the sequence
$$\mathbf{t}_{p,n}'=\sum_{i=0}^n \text{sgn}\left[ {n \choose i} \bmod{p}\right]\, 2^i$$
where the sign function maps any non-zero value to $1$ (and $0$ to $0$). For instance, $\mathbf{t}_{2,n}=\mathbf{t}_{2,n}'$ and the first few terms of $(\mathbf{t}_{3,n}')_{n\ge 0}$ are
$$1, 3, 7, 9, 27, 63, 73, 219, 511, 513, 1539, 3591, 4617,\ldots.$$
In that case, we can directly adapt Equation \eqref{eq: recurrence} where the multiplication by $\binom{n_k}{m}$ is no more necessary, or in Theorem \ref{thm:sec1} where  there is no permutation to consider, and we get the following result.

\begin{proposition} If $n=\sum_{\ell=0}^{k} n_\ell p^\ell$ with $n_k \neq 0$ and if $s=n\bmod p^k$, then we have 
  $$\mathbf{t}_{p,n}'=\mathbf{t}_{p,s}' \sum_{m=0}^{n_k} 2^{mp^k}.$$
\end{proposition}

\section{Polynomial identities}\label{sec: polynomials}

As a preliminary, we state a classical result about formal power series 
$$\prod_{i=0}^\infty (1+X^{p^i}+X^{2p^i}+\cdots+X^{(p-1)p^i})=\sum_{n=0}^\infty X^n.$$
Expanding the left-hand side, we see that this result is equivalent to the fact that every integer has a unique base-$p$ expansion. In this section, we  will often make use of similar arguments to obtain polynomial identities. Our developments are based on \cite{Krizek} (also see \cite{Hewgill}).

We will consider the polynomial rings $\mathbb{Z}[X]$ and $\mathbb{Z}_p[X]$ over $\mathbb{Z}$ and $\mathbb{Z}_p$ respectively. The canonical projection $\pi\colon\mathbb{Z}\to\mathbb{Z}_p : n \mapsto [n]_p$ extends to a ring homomorphism from $\mathbb{Z}[X]$ to $\mathbb{Z}_p[X]$, wich we also denote by $\pi$. Note that the restriction of this map to the subset $\mathbb{Z}_{<p}[X]$ of $\mathbb{Z}[X]$ made of polynomials with coefficients in $\{0,\ldots,p-1\}$ is injective.
Using this notion, we are able to present the following result, which was already noticed on several occasions for $p=2$.
\begin{proposition}\label{lem:poly}
If $n=\sum_{\ell=0}^{k} n_\ell p^\ell$ with $n_k \neq 0$ and if $\prod_{i=0}^k\binom{n_i}{\delta_i}<p$ for all indices $\delta_0,\ldots,\delta_k$, then we have
  \begin{equation}
    \label{eq:poly1}
    \sum_{j=0}^n \left[ {n \choose j} \bmod{p}\right] X^j=\prod_{i=0}^k (1+X^{p^i})^{n_i}
  \end{equation}
in $\mathbb{Z}[X]$.
\end{proposition}

\begin{proof}
We first observe that both polynomials 
\begin{equation}\label{eq:polPnQn}
P_n(X) = \sum_{j=0}^n \left[ {n \choose j} \bmod{p}\right] X^j \quad \text{and} \quad Q_n(X) = \prod_{i=0}^k (1+X^{p^i})^{n_i} 
\end{equation} have the same projection in $\mathbb{Z}_p[X]$. Indeed, on the one hand, we have
\[
\pi\left(P_n(X)\right) = \sum_{j=0}^n \left[\binom{n}{j} \right]_p X^j = \pi\left( (1+X)^n\right).
\]
On the other hand, using that $\pi$ is a ring homomorphism and taking into account the equalities $\pi((1+X)^{p^i}) = \pi(1+X^{p^i})$, for $i \in \mathbb{N}$, we find
\[
\pi(Q_n(X)) = \pi \left(\prod_{i=0}^k \left(1+X^{p^i}\right)^{n_i}\right) =  \pi\left(\prod_{i=0}^k (1+X)^{n_ip^i} \right) 
                     = \pi\left( (1+X)^n\right).
\]
Secondly, we check that the polynomials $P_n(X)$ and $Q_n(X)$ have coefficients in $\mathbb{Z}_{<p}$. Since it is direct for $P_n(X)$, we concentrate on $Q_n(X)$ and find
\[
\prod_{i=0}^k \left(1+X^{p^i}\right)^{n_i} = \prod_{i=0}^k \sum_{\delta_i = 0}^{n_i} \binom{n_i}{\delta_i} X^{\delta_i p^i}
                                                       = \sum_{\delta_k = 0}^{n_k} \ldots \sum_{\delta_0 = 0}^{n_0} \left(\prod_{i=0}^k \binom{n_i}{\delta_i}\right) X^{\delta_k p^k + \ldots + \delta_0}. 
\]
Using the uniqueness of the base-$p$ expansion of every integer together with the assumption of the proposition, we obtain that $Q_n(X)$ belongs $\mathbb{Z}_{<p}[X]$. The conclusion then follows from the injectivity of the restriction of $\pi$ to this set. 
\end{proof}

\begin{remark}\label{rk:condition pour 2}
  For $p=2$, the assumption $\binom{n_k}{\delta_k}\cdots \binom{n_0}{\delta_0}<2$ of Proposition~\ref{lem:poly} always holds because $n_i\le 1$ and thus $\binom{n_i}{\delta_i}\le 1$ for all $i$. Evaluating \eqref{eq:poly1} at $X=2$, we get back $\mathbf{t}_{2,n}=\prod_{i=0}^k(2^{2^i}+1)^{n_i}$, which is~\eqref{eq:LarryRoberts}. 
  
  We even get a family of sequences by evaluating this polynomials identity at other values of $X$. For instance, for $X=3$, the first few terms of the corresponding sequence are
  $$1, 4, 10, 40, 82, 328, 820, 3280, 6562, 26248, 65620,\ldots$$
  and of course, such a sequence $(\mathbf{x}_{n})_{n\ge 0}$ satisfies
  $$\mathbf{x}_{2^i+s}=(3^{2^i}+1)\, \mathbf{x}_s.$$
\end{remark}

\begin{remark}
On the other hand, for $p=3$, $\binom{n_i}{\delta_i}=2$ whenever $n_i=2$ and $\delta_i=1$. In all other situations, the corresponding binomial coefficient is $1$. So $\binom{n_k}{\delta_k}\cdots \binom{n_0}{\delta_0}<3$ holds for all $\delta_i$ if and only if at most one digit $n_i=2$ occurs. Precisely,
$$\binom{n_k}{\delta_k}\cdots \binom{n_0}{\delta_0}=2^{\# \{ i: (n_i,\delta_i)=(2,1)\}}.$$  
Let $p\ge 3$. If the condition in Proposition~\ref{lem:poly}  is not met, then the polynomials $P_n(X)$ and $Q_n(X)$ (defined in Equation~\eqref{eq:polPnQn}) are no longer equal and we have a non-zero difference in $\mathbb{Z}[X]$ expressed as
  $$ \prod_{i=0}^k (1+X^{p^i})^{n_i}-\sum_{j=0}^n \left[ {n \choose j} \bmod{p}\right] X^j$$
  $$=\sum_{\delta_k=0}^{n_k}\cdots \sum_{\delta_0=0}^{n_0} \left\{\binom{n_k}{\delta_k}\cdots \binom{n_0}{\delta_0} -  \left[ {n \choose \val_p(\delta_k\, p^k +\cdots +\delta_0)} \bmod{p}\right]\right\} X^{\delta_k\, p^k +\cdots +\delta_0}.$$
  Let us denote the latter polynomial by $\kappa_{p,n}(X)$. Otherwise stated, we get
  $$\sum_{j=0}^n \left[ {n \choose j} \bmod{p}\right] X^j=\prod_{i=0}^k (1+X^{p^i})^{n_i}-\kappa_{p,n}(X).$$
\end{remark}

\begin{remark}\label{Fp}
  Let us write $n=n_k\, p^k+\cdots +n_1\, p+n_0$ with $n_i \in \{0,\ldots,p-1\}$ for all $0\le i \le k$. Hence $pn+1=n_k\, p^{k+1}+\cdots +n_1\, p^2+n_0\, p+1$. Let us compare the polynomials
  $$\sum_{j=0}^{pn} \left[ {pn \choose j} \bmod{p}\right] X^j \text{ and }
  \sum_{j=0}^{pn+1} \left[ {pn+1 \choose j} \bmod{p}\right] X^j$$
  over $\mathbb{Z}[X]$.
The second one is equal to 
$$(1+X)\prod_{i=0}^k (1+X^{p^{i+1}})^{n_i}-\kappa_{p,pn+1}(X).$$
We can rewrite $\kappa_{p,pn+1}(X)$ as
\begin{multline*}
  \sum_{\delta_k=0}^{n_k}\cdots \sum_{\delta_0=0}^{n_0}\sum_{j=0}^1 \\
  \left\{\binom{n_k}{\delta_k}\cdots \binom{n_0}{\delta_0}\binom{1}{j} -  \left[ {pn+1 \choose \val_p(\delta_k\, p^{k+1} +\cdots +\delta_0\, p+j)} \bmod{p}\right]\right\} \\ X^{\delta_k\, p^{k+1} +\cdots +\delta_0\, p+j}.
\end{multline*}
Since $\binom{1}{j}=1$, Lucas' theorem gives
$${pn+1 \choose \val_p(\delta_k\, p^{k+1} +\cdots +\delta_0\, p+j)} \equiv
{pn \choose \val_p(\delta_k\, p^{k+1} +\cdots +\delta_0\, p)} \pmod{p}.$$
We conclude that $\kappa_{p,pn+1}(X)=(1+X)\cdot \kappa_{p,pn}(X)$. Consequently, we have
\begin{eqnarray*}
  \sum_{j=0}^{pn+1} \left[ {pn+1 \choose j} \bmod{p}\right] X^j&=&(1+X) \cdot \left( \prod_{i=0}^k (1+X^{p^{i+1}})^{n_i}-\kappa_{p,pn}(X)\right)\\
  &=&(1+X) \cdot  \sum_{j=0}^{pn} \left[ {pn \choose j} \bmod{p}\right] X^j.
\end{eqnarray*}
Evaluating this polynomial at $X=p$, we generalize \eqref{F3} to
$$\frac{\mathbf{t}_{p,pn+1}}{\mathbf{t}_{p,pn}}=p+1.$$
 
Notice that we can carry these computations because $\binom{1}{j}=1$. Considering $pn+r$ with $r>1$ is therefore trickier.
\end{remark}
\section{A Nim interlude}\label{sec:nim}

Let $k\ge 2$ be an integer. Recall that a sequence $(x_n)_{n\ge 0}$ of integers is {\em $k$-regular} if the $\mathbb{Z}$-module generated by the set of subsequences
$$\{(x_{k^en+r})_{n\ge 0}\mid e\ge 0, 0\le r<k^e\}$$
is finitely generated, i.e., these subsequences are linear combinations of a finite number of sequences. This notion extends to multidimensional sequences. In particular, a bidimensional sequence $(x_{m,n})_{m,n\ge 0}$ is {\em $k$-regular}  if the $\mathbb{Z}$-module generated by the set of subsequences
$$\{(x_{k^em+r,k^en+s})_{m,n\ge 0}\mid e\ge 0, 0\le r,s<k^e\}$$
is finitely generated. See \cite[Chap.~14,16]{AS}. 

Let us focus again on $\mathbf{t}_{2,n}$. In this section, we show that the sequence $(\mathbf{t}_{2,n})_{n\ge 1}$ naturally appears as a subsequence of a well-studied $2$-regular sequence that we denote by $(N(m))_{m\ge 0}$.  For the sake of presentation, we limit ourselves to the case $p=2$ but a similar discussion can be carried on for any modulo.

We let $m\oplus n$ denote the Nim-sum of the integers $m,n$, i.e., addition digit-wise modulo~$2$ of their base-$2$ expansions (without carry). For instance, $5\oplus 12=9$. From Pascal's rule $\binom{n+1}{i}=\binom{n}{i}+\binom{n}{i-1}$, we have that
\begin{equation}
  \label{eq:z2n}
  \mathbf{t}_{2,n+1}=\mathbf{t}_{2,n}\oplus (2 \mathbf{t}_{2,n}).
\end{equation}
For all $r,s\in\{0,1\}$, we have
$$(2m+r)\oplus (2n+s)=2(m\oplus n) + (r+s \bmod{2}),$$
which means that the bidimensional sequence $(m\oplus n)_{m,n\ge 0}$ is $2$-regular; see \cite[Example~16.5.5]{AS}.
In view of relation \eqref{eq:z2n}, consider the subsequence $(N(m))_{m\ge 0}$ extracted from the previous bidimensional Nim-sum array and defined by $N(m)=m\oplus 2m$ for all $m\ge 0$.
The sequence $(N(m))_{m\ge 0}$ starts with values given in Table~\ref{tab:Nm}, which are also depicted in Figure~\ref{fig:nim2i}, and appears as entry {\tt A048724} in the OEIS \cite{Sloane} (it also appears in \cite{Nguyen}).
\begin{table}[h!tb]
 $$\begin{array}{r|cccccccccccccccccc}
    m & 0 & 1 & 2 & 3 & 4 & 5 & 6 & 7 & 8 & 9 & 10 & 11 & 12 & 13 & 14 & 15 & 16 & 17\\
    \hline
N(m)& 0& \underline{3}& 6& \underline{5}& 12& \underline{15}& 10& 9& 24& 27& 30& 29& 20& 23& 18& \underline{17}& 48& \underline{51}
  \end{array}$$
  \caption{The sequence $(N(m))_{m\ge 0}$.}
  \label{tab:Nm}
\end{table}
\begin{figure}[h!tb]
  \centering
  \includegraphics[height=4.5cm]{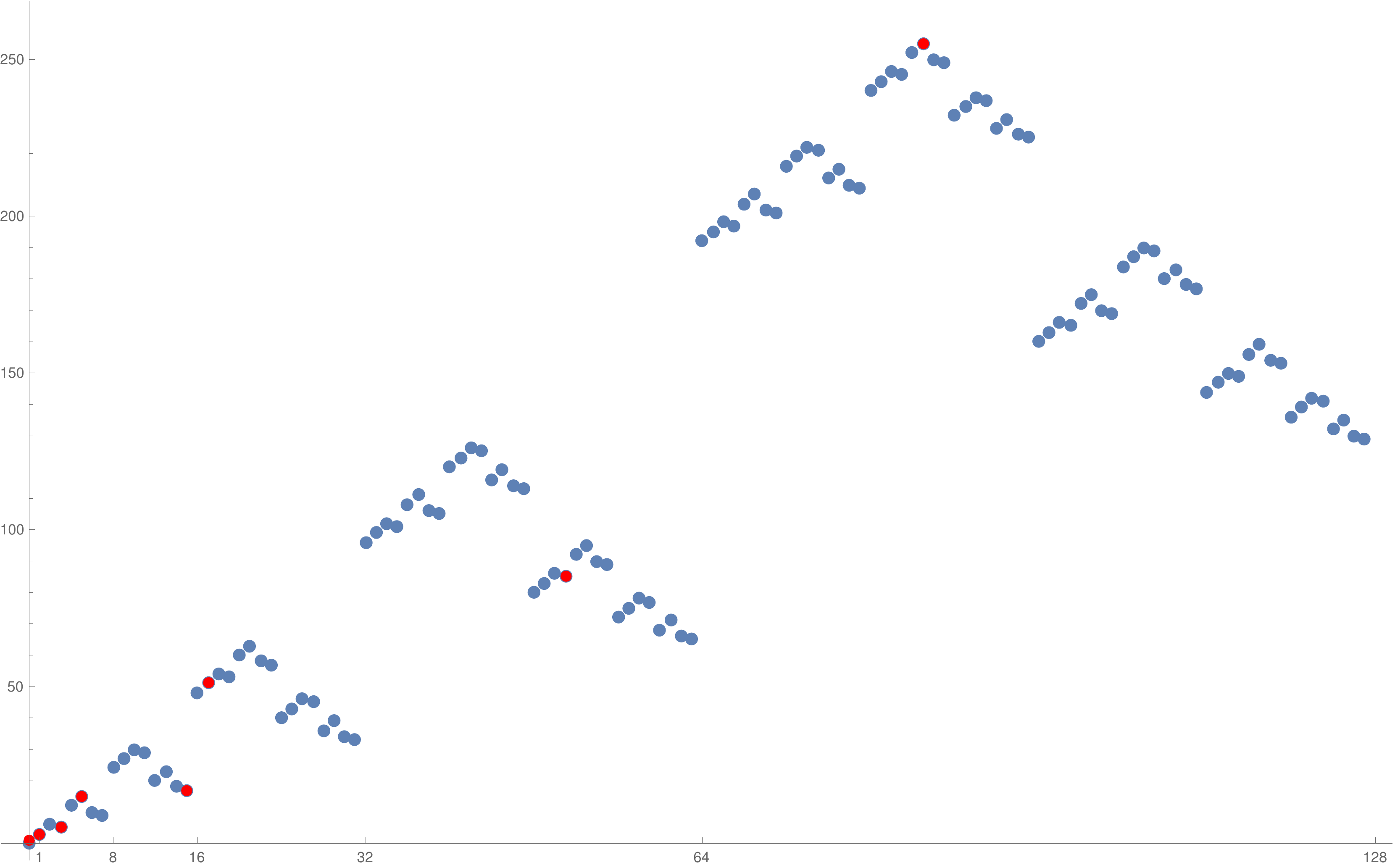}
  \caption{The first few values of $(N(m))_{m\ge 0}$.}
  \label{fig:nim2i}
\end{figure}
In Figure~\ref{fig:nim2i}, observe that a pattern repeats itself between two consecutive powers of $2$, suggesting that the considered sequence is $2$-regular. This property is shown below.
The red dots in Figure~\ref{fig:nim2i} represent the first few values of the subsequence $(\mathbf{t}_{2,n})_{n\ge 0}$ of $N(m))_{m\ge 0}$. Since $|\rep_2(\mathbf{t}_{2,n})|=n+1$, note that there is only one red dot in an interval made of consecutive powers of $2$. 
Playing with base-$2$ expansions, it is easily seen that any subsequence of the form $(N(4m+r))_{m\ge 0}$, $0\le r<4$, can be expressed as a linear combination of the three sequences
$$(N(m))_{m\ge 0},\ (N(2m+1))_{m\ge 0},\ (1,1,1,1,\ldots).$$
Indeed we have
\begin{equation}
  \label{eq:recN4}
\left\{
  \begin{array}{rcl}
    N(4m)&=&4N(m), \\ 
    N(4m+1)&=&4N(m)+3,\\ 
    N(4m+2)&=&2N(2m+1),\\ 
    N(4m+3)&=&2N(2m+1)-1.\\
  \end{array}\right.   
\end{equation}
For instance, the first relation holds as $4m \oplus 8m = 4(m \oplus 2m)$ since $\rep_2(8m)=\rep_2(m)000$ and $\rep_2(4m)=\rep_2(m)00$.
Now let us come back to the sequence $(\mathbf{t}_{2,n})_{n\ge 1}$ of interest. From \eqref{eq:z2n}, this is a subsequence of $(N(m))_{m\ge 0}$. Namely, $\mathbf{t}_{2,0}=1$ and, for all $n>0$,
\begin{equation}
  \label{eq:ext}
  \mathbf{t}_{2,n}=N(\mathbf{t}_{2,n-1}).
\end{equation}
In Table~\ref{tab:Nm}, $(\mathbf{t}_{2,n})_{n\ge 1}$ appears underlined. In the next section, we focus on this kind of subsequence extraction.

\begin{remark}
Note that the same argument can be carried out for a general prime $p$ as long as one defines a suitable Nim-sum in base $p$: addition digit-wise modulo~$p$ (without carry). For instance, $23\oplus_313=6$. We have $\mathbf{t}_{p,n+1}=\mathbf{t}_{p,n}\oplus_p(p\, \mathbf{t}_{p,n})$ and we may define $N_p(m)=m\oplus_p (p\, m)$ to get $\mathbf{t}_{p,n}=N_p(\mathbf{t}_{p,n-1})$.
In the following, we keep the notation $N(m)$ for $N_2(m)$.
\end{remark}

\subsection{Regularity of $(N_{p}(m))_{m\ge 0}$ and finite automata} 
We assume that the reader has some basic knowledge of automata theory. A (deterministic finite) automaton is a machine devised to recognize/accept some sequences of symbols read once at a time. In our setting, these symbols are usually pairs or tuples of digits. See for instance \cite{AS,BHMV} for some background on the matter.

In the remaining of this section, we show that the sequence $(N_p(m))_{m\ge 0}$ is $p$-synchronized but that the sequence $(\mathbf{t}_{p,n})_{n\ge 0}$ is not $p$-regular. We recall the necessary definitions. Let $d\ge 1$ and $k\ge 2$ be integers. A subset $X$ of $\mathbb{N}^d$ is {\em $k$-recognizable} (or said to be a {\em $k$-synchronized relation} with the terminology of \cite{Carpi}) if the language
$$\{{\tt pad}(\rep_k(x_1),\ldots,\rep_k(x_d))\mid (x_1,\ldots,x_d)\in X\}$$
  is accepted by a finite automaton with input alphabet $\{0,\ldots,k-1\}^d$ and where ${\tt pad}(m_1,\ldots,m_d)$ is the $d$-tuple of words of the same length
  $$\left(0^{M-|m_1|} m_1,\ldots, 0^{M-|m_d|} m_d\right)$$
  with $M=\max_i |m_i|$. A sequence $(x_n)_{n\ge 0}$ is {\em $k$-synchronized} (see \cite{Carpi}) if the set $\{(n,x_n)\mid n\ge 0\}$ is $k$-recognizable. Every $k$-synchronized sequence is $k$-regular \cite[Prop.~2.6]{Carpi}.

  \begin{proposition}
  Let $a,b$ be non-negative integers. The set of pairs $\{(m,am+b)\mid m\ge 0\}$ is $p$-recognizable. Otherwise stated, the sequence $(am+b)_{m\ge 0}$ is $p$-synchronized.
  \end{proposition}

  \begin{proof}
    This is a classical exercise in automata theory or, one can make use of the fact that this set of pairs is $p$-definable (i.e., definable by a first order formula in $\langle \mathbb{N},+,V_p\rangle$) --- see, for instance, \cite{BHMV}. Indeed, multiplication by a constant is definable in this structure. 
  \end{proof}

  \begin{proposition}
    The set of triples $\{(m,n,m\oplus_p n)\mid m,n\ge 0\}$ is $p$-recognizable.
  \end{proposition}

  \begin{proof}
    A single-state automaton with a loop of labels $(a,b,a+b\bmod{p})$ is enough. There is no carry to take into account. 
  \end{proof}

Composing synchronized relations \cite{Carpi}, we get the following.
  
  \begin{corollary}
    Let $a,b$ be non-negative integers. The sequence $(m\oplus_p(am+b))_{m\ge 0}$ is $p$-synchronized. In particular, $(N_p(m))_{m\ge 0}$ is $p$-synchronized.
  \end{corollary}

  \begin{proof}
    Combining the above two propositions, the set
    $$\{(m,am+b,m\oplus_p (am+b))\mid m\ge 0\}$$
    is $p$-recognizable. 
  \end{proof}

  In Figure~\ref{fig:DFAz2n}, we have represented an automaton recognizing $\{(m,N(m))\mid m\ge 0\}$, where all transitions leading to a sink state are not drawn. The first few pairs of words that are accepted are
  $$\binom{\varepsilon}{\varepsilon},\ \binom{01}{11},\ \binom{010}{110},\ \binom{011}{101},\ \binom{0100}{1100},\ \binom{0101}{1111}$$
  which correspond to the pairs of integers $(0,0)$, $(1,3)$, $(2,6)$, $(3,5)$, $(4,12)$, $(5,15)$.
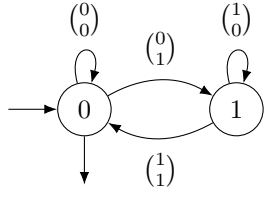
\begin{figure}[htb]
\begin{center}
\begin{tikzpicture}
\tikzstyle{every node}=[shape=circle, fill=none, draw=black,minimum size=20pt, inner sep=2pt]
\node(1) at (0,0) {$0$};
\node(2) at (2,0) {$1$};

\tikzstyle{every node}=[shape=circle, minimum size=5pt, inner sep=2pt]

\draw [-Latex] (-1,0) to node [above] {} (1);
\draw [-Latex] (1) to node [above] {} (0,-1);

\draw [-Latex] (1) to [loop above] node [above] {$\binom{0}{0}$} (1);
\draw [-Latex] (1) to [bend left] node [above] {$\binom{0}{1}$} (2);
\draw [-Latex] (2) to [loop above] node [above] {$\binom{1}{0}$} (2);
\draw [-Latex] (2) to [bend left] node [below] {$\binom{1}{1}$} (1);

\end{tikzpicture}
\end{center}
\caption{A DFA recognizing the pairs $(m,N(m))$.}
\label{fig:DFAz2n}
\end{figure}
For example, an accepting run for the pair $(4,12)$, starting from the initial state $0$, is given by
$$0\stackrel{\binom{0}{1}}{\longrightarrow}1\stackrel{\binom{1}{1}}{\longrightarrow}0\stackrel{\binom{0}{0}}{\longrightarrow}0\stackrel{\binom{0}{0}}{\longrightarrow}0.$$

From the classical theory of regular sequences, we can also obtain a linear representation for $(N(m))_{m\ge 0}$:
$$\lambda=
\begin{pmatrix}
  1&0&0\\
\end{pmatrix},\, 
\mu(0)=
\begin{pmatrix}
  2&0&0\\
  0&0&1\\
  4&0&1\\
\end{pmatrix}, \,
\mu(1)=
\begin{pmatrix}
  0&1&0\\
  4/3&2&-1/3\\
  -4&4&1\\
\end{pmatrix}, \,
\nu=
\begin{pmatrix}
  0\\ 3\\ 3\\
\end{pmatrix}.$$
This means that $N(m)$ can be computed as $\lambda\cdot \mu(\rep_2(m)^R)\cdot \nu$ where $\mu$ is a morphism from the monoid $\{0,1\}^*$ equipped with concatenation to the monoid $\mathbb{Z}^{3\times 3}$ equipped with multiplication. Matrix multiplications are considered starting with the least significant digit first or, with the reversal of the base-$2$ expansion of $m$. For instance, $\rep_2(4)=100$ and $\lambda\cdot \mu(0)\cdot \mu(0)\cdot \mu(1)\cdot \nu=12=N(4)$.

  \begin{proposition}\label{pro:notreg}
    The sequence $(\mathbf{t}_{p,n})_{n\ge 0}$ is not $p$-regular.
  \end{proposition}
  \begin{proof}
    If a sequence is $p$-regular then its growth rate is in $\mathcal{O}(n^c)$ for some constant $c$. But from \eqref{F3}, $\mathbf{t}_{2,n+4}\ge 9\, \mathbf{t}_{2,n}$ and thus $\mathbf{t}_{2,n}\ge 9^{n/4}$. More generally, for an arbitrary $p\ge 2$, from Remark~\ref{Fp}, $\mathbf{t}_{p,n+2p}\ge (p+1)^2\, \mathbf{t}_{p,n}$ and thus $\mathbf{t}_{p,n}\ge (p+1)^{n/p}$. 
  \end{proof}

  \section{The set $\{N_p(m)\mid m\ge 0\}$}\label{sec:par}
Throughout this section, we let $p\ge 2$ be a prime number. Our goal is to study the set $\{N_p(m)\mid m\ge 0\}$.

\begin{lemma}\label{lem:injective}
The map $m\mapsto N_p(m)$ is injective. 
\end{lemma}

\begin{proof}
  Let $x,y$ be such that $\rep_p(x)=x_\ell\cdots x_0$, $\rep_p(y)=y_\ell \cdots y_0$. If the two representations have different lengths, we allow leading zeroes for the shortest one. Assume $x\neq y$. Let $k\ge 0$ be the smallest index such that $x_k\neq y_k$. Then
  $x_k\oplus_p x_{k-1}=x_k\oplus_p y_{k-1}$ differs from $y_k\oplus_p y_{k-1}$, so $N_p(x)\neq N_p(y)$. 
\end{proof}

\subsection{Partitioning $\{N(m) \mid m\ge 1\}$}
We start with the case $p=2$ and we show that $\{N(m)\mid m\ge 0\}$ may be partitioned into sets of numbers obtained by recursively iterating the map $m\mapsto N(m)$ on odious numbers. An {\em evil} (resp. {\em odious}) number is an integer having an even (resp. odd) number of $1$'s in its base-$2$ expansion.

  \begin{lemma}\label{lem:uniqueEvil}
    Let $e$ be an evil number. There is a unique integer $m$ such that $N(m)=e$.
  \end{lemma}

  \begin{proof}
        Let $e$ be an evil number, and write $\rep_2(e)=e_\ell\cdots e_0$ with $e_\ell=1$. We show that there exists $m$ such that $\rep_2(m)=m_{\ell-1}\cdots m_0$ with $m_{\ell-1}=1$ and $N(m)=e$. If such an integer $m$ exists, then it must satisfy $m_0=e_0$, $m_{\ell-1}=e_\ell=1$ and we find $m_i=e_i\oplus m_{i-1}$ for $i=1,\ldots,\ell$ (if, for convenience, we set $m_\ell=0$). Otherwise stated, $\rep_2(m)$ is made of blocks of $0$'s or $1$'s. If $e_i=1$, then $m_i=1-m_{i-1}$, so these two kinds of blocks alternate each time we encounter a letter $1$ in the base-$2$ expansion of $e$. Starting from the least significant digit, the rightmost block in $\rep_2(m)$ is made of letters $e_0$. Since $e$ is evil, with $e_\ell=1$, we indeed get $m_\ell=0$ (we thus have a solution to the system of equations). Uniqueness follows from the previous lemma.
  \end{proof}
  
  \begin{example}
  In the proof of the previous lemma, we start with an evil number.
  Let us take $e=43$ with $\rep_2(n)=101011$.
  Then we would like to find the solution to the equation $N(m)=e$, which is represented in the following table.
  \[
  \begin{array}{c|ccccccc}
  \rep_2(m) & & 0 & m_4 & m_3 & m_2 & m_1 & m_0 \\
  \rep_2(2m) & \oplus & m_4 & m_3 & m_2 & m_1 & m_0 & 0 \\
  \hline
  \rep_2(e) & & 1 & 0 & 1 & 0 & 1 & 1
  \end{array}
  \]
  Starting from the right of the table, we get $m_0=1$.
  We may update the table as follows.
    \[
  \begin{array}{c|ccccccc}
  \rep_2(m) & & 0 & m_4 & m_3 & m_2 & m_1 & 1 \\
  \rep_2(2m) & \oplus & m_4 & m_3 & m_2 & m_1 & 1 & 0 \\
  \hline
  \rep_2(e) & & 1 & 0 & 1 & 0 & 1 & 1
  \end{array}
  \]
  Examining the second column on the right, we get $m_1=0$.
  Pursuing like this for the other columns, we obtain $\rep_2(m)=11001$.
   \end{example}
  
  \begin{lemma}\label{lem:3}
    The set $\{N(m)\mid m\ge 0\}=\{0,3,5,6,9,10,\ldots\}$ is exactly the set of evil numbers. In particular, the sequence $(N(m))_{m\ge 1}$ is a permutation of the increasing sequence {\tt A001969} of evil numbers.
  \end{lemma}

  \begin{proof}
    We first show by induction on $m$ that $N(m)$ is evil. This is readily checked for the first few values of $m$. We make use of \eqref{eq:recN4}: $\rep_2(N(4m))=\rep_2(N(m))00$, $\rep_2(N(4m+1))=\rep_2(N(m))11$ and $\rep_2(N(4m+2))=\rep_2(N(2m+1))0$. By induction hypothesis, $\rep_2(N(m))$ and $\rep_2(N(2m+1))$ are evil, so are $N(4m+r)$ for $r=0,1,2$. Observe that $N(2m+1)$ is odd by definition of the Nim-sum. Thus $\rep_2(N(2m+1))$ is of the form $u1$ for some binary word $u$. Hence, $\rep_2(2N(2m+1)-1)=u01$ has the same number of letters $1$ as $\rep_2(N(2m+1))$. By induction hypothesis, $N(2m+1)$ is evil and thus $N(4m+3)$ is also evil.

    Conversely, every evil number belongs to the set as a consequence of the previous lemma. 
  \end{proof}
  
Since $N(m)>m$ for all $m\ge 1$, we can extract subsequences in a recursive way similar to \eqref{eq:ext}. Let $i\ge 1$ be an integer. We let $({\tt sub}(i,m))_{m\ge 0}$ be the sequence $(x_m)_{m\ge 0}$ defined by $x_0=i$ and $x_{m+1}=N(x_m)$. In other words we consider the sequence $(N^m(i))_{m\ge 0}$ of iterations of $N$ on $i$. We have seen in particular that $\mathbf{t}_{2,m}={\tt sub}(1,m)$ for all $m\ge 1$ (recall, for instance, Table~\ref{tab:Nm}).

We claim the following.
\begin{theorem}
We have the partition
$$\{N(m)\mid m\ge 1\} = \bigcup_{i\in \mathcal{O}} \{ {\tt sub}(i,m) \mid m\ge 1\}$$
where the sets in the above union are pairwise disjoint and $\mathcal{O}=\{1,2,4,7,8,11,\ldots\}$ is the set of odious numbers.
\end{theorem}

$$\begin{array}{c|ccccccccc}
    & 1&2&3&4&5&6&7&8& \\
    \hline
 {\tt sub}(1,m) & 3 & 5 & 15 & 17 & 51 & 85 & 255 & 257 & \cdots\\
 {\tt sub}(2,m) & 6 & 10 & 30 & 34 & 102 & 170 & 510 & 514 \\
 {\tt sub}(4,m) & 12 & 20 & 60 & 68 & 204 & 340 & 1020 & 1028 \\
 {\tt sub}(7,m) & 9 & 27 & 45 & 119 & 153 & 427 & 765 & 1799 \\
 {\tt sub}(8,m) & 24 & 40 & 120 & 136 & 408 & 680 & 2040 & 2056 \\
    {\tt sub}(11,m) & 29 & 39 & 105 & 187 & 461 & 599 & 1785 & 2827 \\
    \vdots & \\
  \end{array}$$

  With the same reasoning as in the proof of Proposition~\ref{pro:notreg} from \eqref{F3}, none of these sequences is $2$-regular.

  \begin{proof}
    Let $i\neq j$ be odious numbers. The sets $\{ {\tt sub}(i,m) \mid m\ge 1\}$ and $\{ {\tt sub}(j,m) \mid m\ge 1\}$ are disjoint. Proceed by contradiction and assume that there exist integers $m,n\ge 1$ such that $N^m(i)=N^n(j)$. Without loss of generality, assume $m\ge n$. From Lemma~\ref{lem:injective}, $N^{m-n}(i)=j$. If $m=n$, we get $i=j$, which is a contradiction. If $m>n$, then $N^{m-n}(i)$ is evil but $j$ is odious, which is again a contradiction.

    We still have to show that for every evil number $n$, there exists some integer $i$ such that $n$ belongs to $\{ {\tt sub}(i,m) \mid m\ge 1\}$. If a number $e$ is evil, then find $N^{-1}(e)$ and repeat this procedure while the result is evil and positive. Since $N^{-1}(e)<e$, this procedure stops when we reach an odious number $i$ meaning that $e$ belongs to $\{ {\tt sub}(i,m) \mid m\ge 1\}$.

    Finally, by Lemma~\ref{lem:3} every evil number appears in the set, from the first part of the proof, the partition must thus runs over all odious numbers.
  \end{proof}
  
\subsection{A known permutation}\label{sec:permut}  Since $(N(m))_{m\ge 1}$ is a permutation of the sequence {\tt A001969} of the evil numbers, it is natural to consider the sequence $\alpha$ mapping $m\ge 1$ to the position of $N(m)$ within the ordered sequence of evil numbers. The first few terms of this permutation $\alpha$ of $\mathbb{N}_{>0}$ are
  $$1, 3, 2, 6, 7, 5, 4, 12, 13, 15, 14, 10, 11, 9, 8,$$
  $$24, 25, 27, 26, 30, 31, 29, 28, 20, 21, 23, 22, 18, 19, 17, 16,\ldots.$$
  Otherwise stated, if $e(m)$ is the $m$th evil number, the first evil numbers being $e(0)=0$ and $e(1)=3$, then
  \begin{equation}
    \label{eq:alpha}
    e(\alpha(m))=N(m).
  \end{equation}
  This sequence appears as {\tt A003188} in \cite{Sloane} and is described as an integer equivalent of the Gray code for $n$ considered as a base-$2$ expansion (Gray code provides a way to enumerate integers by only changing one digit in their base-$2$ expansion from one element to the next one). As observed by Paul D. Hanna (we refer again to \cite{Sloane}), it is known that 
 \begin{equation}
    \label{eq:alphadef}
    \alpha(m)=m\oplus \lfloor m/2\rfloor
    \end{equation}  
    for all $m\ge 1$. In the following we define $\alpha$ through this relation. In particular, $N(m)$ is roughly twice $\alpha(m)$. One can easily deduce that the map $\alpha$ restricted to $[2^n,2^{n+1}[$ is again a one-to-one correspondence mapping $[2^n,2^n+2^{n-1}[$ to $[2^n+2^{n-1},2^{n+1}[$ and $[2^n+2^{n-1},2^{n+1}[$ to $[2^n,2^n+2^{n-1}[$, as shown in Figure~\ref{fig:alpha}.
  \begin{figure}[h!tb]
    \centering
    \includegraphics[height=4.5cm]{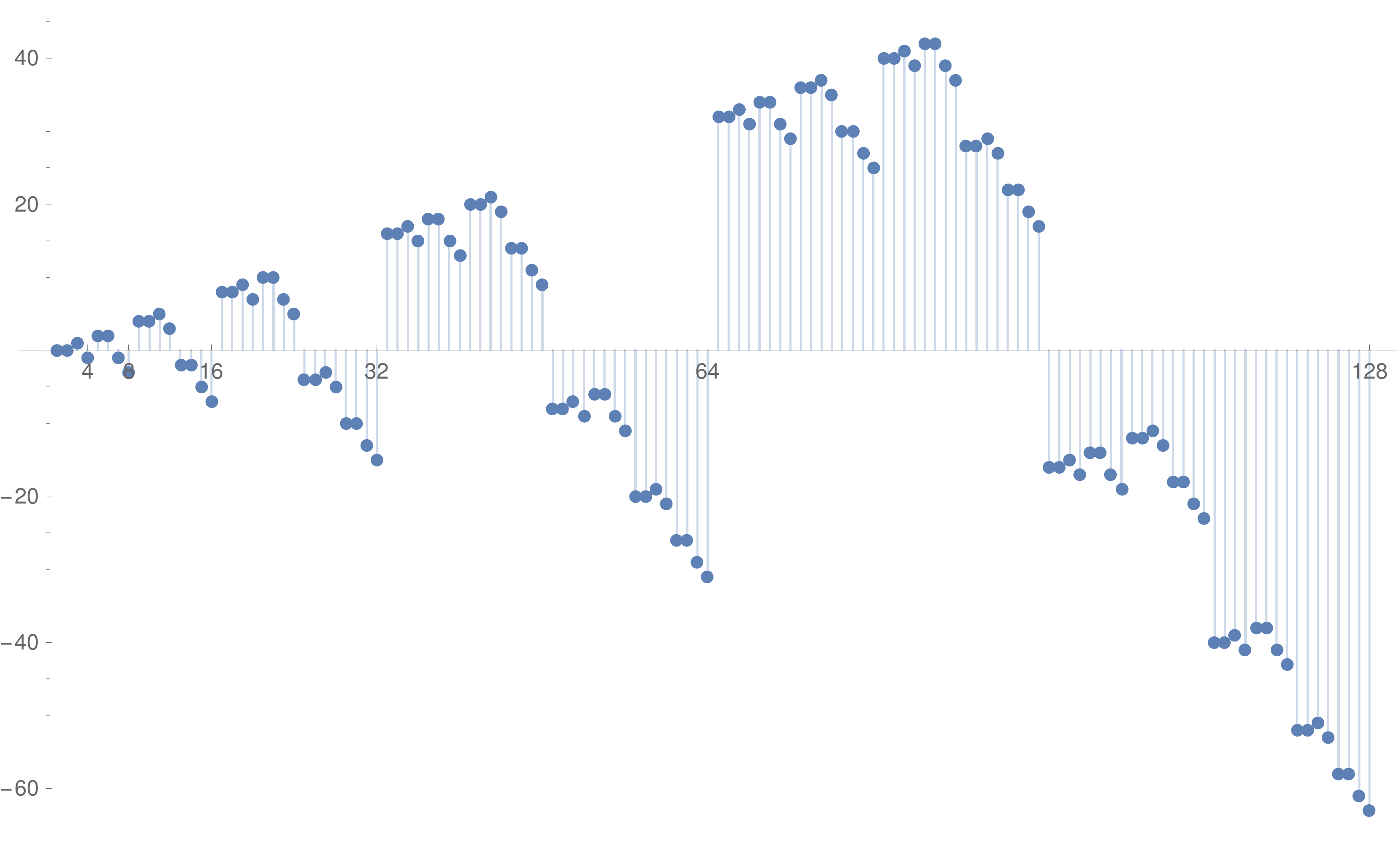}
    \caption{The graph of $n\mapsto \alpha(n)-n$ over $[0,127[$.}
    \label{fig:alpha}
  \end{figure}
  From Hanna's remark~\eqref{eq:alphadef}, we have
  \begin{equation}
    \label{eq:alpha2}
  \left\{\begin{array}{rcl}
    \alpha(4n)&=&2\alpha(2n)\\
             \alpha(4n+1)&=&2\alpha(2n)+1\\
             \alpha(4n+2)&=&2\alpha(2n+1)+1\\
             \alpha(4n+3)&=&2\alpha(2n+1).
         \end{array}\right.
     \end{equation}
      
To get $e(n)$, notice that one simply writes down $\rep_2(n)$ and appends an extra digit, either $0$ or $1$ to get an evil number. This is rather straightforward: indeed $\rep_2(n)$ ranges over all the words in $1\{0,1\}^*$, and appending the convenient digit, we get all the evil numbers (and the order is preserved). 
 If $\rep_2(\alpha(m))=a_\ell \cdots a_1$, then relation~\eqref{eq:alphadef} yields $\rep_2(N(m))=\rep_2(2m\oplus m)=a_\ell\cdots a_1 a_0$ where $a_0$ is the least significant digit of $m$. Moreover it is the only evil number having $a_\ell \cdots a_1$ as length-$\ell$ prefix. From these observations, we get \eqref{eq:alpha} that can be expressed by   
  $$\rep_2(N(m))=\left\{
    \begin{array}{ll}
      \rep_2(\alpha(m))0, & \text{if }\alpha(m)\text{ is evil};\\
      \rep_2(\alpha(m))1, & \text{if }\alpha(m)\text{ is odious}\\
    \end{array}\right.
  $$
  (recall that $N(m)$ is evil).

  \subsection{A generalization of the Thue--Morse sequence}

  In this section, we show that the set $E_p=\{N_p(m)\mid m\ge 0\}$, generalizing the set of evil numbers for $p>2$, is $p$-automatic.
  Recall that a set $S$ of non-negative integers is said to be $p$-automatic if its characteristic sequence
  \[
  \chi_S(n) 
  = 
  \begin{cases}
  1 & \text{if $n\in S$}\\
  0 & \text{otherwise}
  \end{cases}
  \]
  is itself $p$-automatic~\cite{AS}. For more details about automaticity, we refer the reader to \cite{AS} or \cite{BHMV}.

  \begin{lemma}\label{lem: charact-Np}
    Let $e$ be an integer and write $\rep_p(e)=e_{k+1}\cdots e_0$. There exists an integer $m$ such that $N_p(m)=e$ if and only if $\sum_{i=0}^{k+1}(-1)^ie_i= 0\pmod{p}$. When such an integer $m$ exists, it is unique.
  \end{lemma}

  \begin{proof}
    Let $\rep_p(m)=m_k\cdots m_0$. As in the proof of Lemma~\ref{lem:uniqueEvil}, we have to consider the following linear system over $\mathbb{Z}_p$
    $$
    \begin{pmatrix}
      1&0&0&\cdots &0\\
      1&1&0&       &0\\
      0&1&1&       &0\\
       \vdots& &\ddots& \ddots  &\vdots \\ 
       &&&           1&1\\
       0&&\cdots&0&1
    \end{pmatrix}
    \begin{pmatrix}
      m_0\\
      \vdots\\
      m_k\\
    \end{pmatrix}=\begin{pmatrix}
      e_0\\
      \vdots\\
      e_{k+1}\\
    \end{pmatrix}.
    $$
    The $(k+2)\times(k+1)$ matrix has rank $k+1$. The system has a solution if and only the determinant
    $$
    \begin{pmatrix}
      1&0&0&\cdots &0&e_0\\
      1&1&0&       &0&e_1\\
      0&1&1&       &0&e_2\\
       \vdots& &\ddots& \ddots  &\vdots \\ 
       &&&           1&1&e_k\\
       0&&\cdots&0&1&e_{k+1}\\
     \end{pmatrix}$$
     is zero in $\mathbb{Z}_p$. Uniqueness follows from Lemma~\ref{lem:injective}.
  \end{proof}

A classical generalization of the Thue--Morse sequence to a $p$-letter alphabet is to consider the fixed point starting with $0$ of the morphism over $\{0,\ldots,p-1\}$ defined by $i\mapsto i (i+1)\cdots (p-1)\, 0\cdots (i-1)$. The $n$th symbol occurring in the fixed point is equal to the sum-of-digits modulo~$p$ of $n$ written in base~$p$. See, for instance, \cite{ubi} and the references therein.
  
  \begin{proposition}
  Let $\varphi$ be the $p$-uniform morphism over $\{0,\ldots,p-1\}$ defined by $\varphi(0) = 0\, (p-1)\, (p-2)\cdots 1$ and $\varphi(j)= (p-j)\, (p-j-1) \cdots 0\, (p-1)\, (p-2) \cdots (p-j+1)$ for all $j\in\{1,\ldots,p-1\}$, and let $\tau$ be the coding over $\{0,\ldots,p-1\}$ defined by $\tau(0)=1$ and $\tau(j)=0$ for all $j>1$.
  Then the set $E_p=\{N_p(m)\mid m\ge 0\}$ is $p$-automatic, i.e., its characteristic sequence is the image, under the coding  $\tau$, of the fixed point of the morphism $\varphi$.
  \end{proposition}
\begin{proof}
Consider a DFA with $2p$ states of the form $(i,+)$ or $(i,-)$ with $i\in\{0,\ldots,p-1\}$. The transitions between states are given by
    $$(i,+)\stackrel{d}{\longrightarrow} (i+d\bmod{p},-)$$
    and
    $$(i,-)\stackrel{d}{\longrightarrow} (i-d\bmod{p},+)$$
    for all digits $d\in\{0,\ldots,p-1\}$.
    The initial state is $(0,+)$ and the final states are $(0,+)$ and $(0,-)$. This DFA accepts words (i.e., finite sequences of digits) whose alternating sum equals $0$ modulo~$p$. We can minimize this DFA. For $0\le i<p$, the states $(i,+)$ and $(p-i\bmod{p},-)$ are Nerode equivalent, i.e., the same sequences are accepted from both states. Indeed, reading $d_0\cdots d_k$ from $(i,+)$ leads to a state whose first component is $i+d_0-d_1+\cdots +(-1)^kd_k=0$ modulo~$p$. Reading the same word $d_0\cdots d_k$ but from $(p-i,-)$ leads to $p-i-d_0+d_1+\cdots-(-1)^kd_k$, which is also equal to $0$ modulo~$p$. After merging states, the minimal automaton has $p$ states of the form $[(i,+),(p-i\bmod{p},-)]$ for $0\le i< p$ and transitions
    $$[(i,+),(p-i\bmod{p},-)]\stackrel{d}{\longrightarrow} [(p-i-d\bmod{p},+),(i+d\bmod{p},-)]$$
        for all digits $d\in\{0,\ldots,p-1\}$. 
If we identify $[(j,+),(p-j\bmod{p},-)]$ with $j$, we get the expected morphism using a classical construction due to Cobham. For instance, see \cite[Theorem 6.3.2]{AS}.

Now, if $\rep_p(n)=n_k\cdots n_0$, observe that reading the word $n_k\cdots n_0$ from the state $(0,+)$ leads to the state $(n_0-n_1+\cdots +(-1)^k n_k, (-1)^{k+1})$.
We conclude the proof by using Lemma~\ref{lem: charact-Np}.
\end{proof} 

We can make the same discussion as in Subsection~\ref{sec:permut}. In an attempt to generalize~\eqref{eq:alpha}, we extend~\eqref{eq:alphadef} by defining $\alpha_p(m):=m\oplus \lfloor m/p\rfloor$ and by letting $E_p(m)$ denote the $m$th element in $E_p$.
It is clear that $\{\alpha_p(m)\mid m\ge 1\}=\mathbb{N}_{>0}$, and thus, $\alpha_p$ is a permutation of $\mathbb{N}_{>0}$. 
For instance, for $p=3$, the first few terms of $\alpha_3$ \cite[{\tt A071770}]{Sloane} are
$$0, 1, 2, 4, 5, 3, 8, 6, 7, 12, 13, 14, 16, 17, 15, $$
$$11, 9, 10, 24, 25, 26, 19, 20, 18, 23, 21, 22, 9, 37, 38, 40, 41,\ldots.$$ 
 For every integer $m\ge 0$ such that $\rep_p(\alpha_p(m))=a_\ell\cdots a_1$, there exists a unique digit $a_0$ such that $\val_p(a_\ell\cdots a_1 a_0)$ belongs to $E_p$ by Lemma~\ref{lem: charact-Np}. This is the $m$th element in $E_p$; thus $E_p(\alpha_p(m))=N_p(m)$. 

  \subsection{Summatory function}

  Jean-Paul Allouche {\em et al.} \cite{All} provide an exact formula for the summatory function of the evil numbers (they also consider the generalization to arbitrary bases and digits)
  \begin{eqnarray*}
    S_e(M)&:=&\sum_{i=1}^M e(i)=M(M+1)+\left\lfloor \frac{M}{2}\right\rfloor- \frac{[M \% {2}]([M \% {2}]+1)}{2}\\
 && +[s_2(\lfloor M/2\rfloor)\%{2}] ([M \% {2}]+1)
  +2 \max\bigl\{0,[M \% {2}]-[s_2(\lfloor M/2\rfloor)\% {2}]\bigr\},
  \end{eqnarray*}
  where we let $s_2$ denote the sum-of-digits in base~$2$ and $[n\% k]$ denote the unique integer in $\{0,\ldots,k-1\}$ congruent to $n$ modulo $k$. In this formula, the first two terms explain the general behavior and last three terms only give a possible correction of $1$ to the main terms. 
  In the same vein, let us consider the summatory function of $N$ given by
  $$S_N(M):=\sum_{i=1}^M N(i).$$
Instead of considering general/advanced techniques on the summatory function of $k$-regular sequences \cite[Section~3.5]{AS}, we will make use of elementary operations and of the permutation $\alpha$ to express $S_N(M)$.
Now let $k\ge0$ be an integer. Because of \eqref{eq:alpha}, $e(\cdot)$ and $N(\cdot)$ take the same set of values over any interval of the form $[2^k,2^{k+1}[$, thus 
  we have $S_N(2^k-1)=S_e(2^k-1)$. However for $M\in [2^k,2^{k+1}-1[$, $S_N(M)>S_e(M)$. The graph of the difference between $S_N$ and $S_e$ is given in Figure~\ref{fig:diffSN}.
  \begin{figure}[h!tb]
    \centering
    \includegraphics[height=4.5cm]{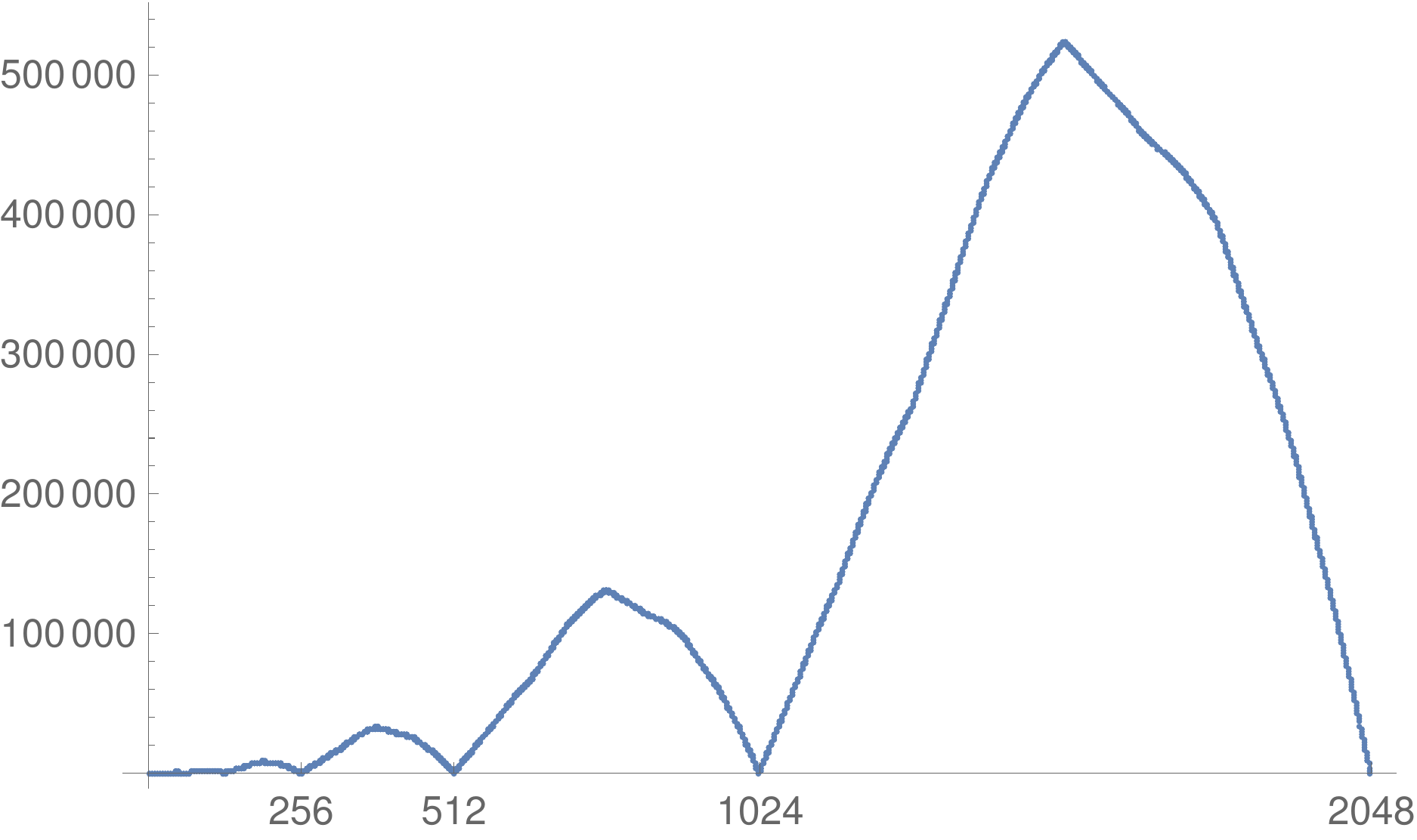}
    \caption{The difference $M\mapsto S_N(M)-S_e(M)$.}
    \label{fig:diffSN}
  \end{figure}
  Taking into account the behavior of the permutation $\alpha$ on the interval $[2^k,2^{k+1}[$, we obtain
  $$\sum_{j=2^k}^{2^k+2^{k-1}-1}N(j)=\sum_{j=2^k+2^{k-1}}^{2^{k+1}-1}e(j)\quad \text{ and }\quad
  \sum_{j=2^k+2^{k-1}}^{2^{k+1}-1}N(j)=\sum_{j=2^k}^{2^k+2^{k-1}-1}e(j).$$
  Since $e$ is an increasing sequence, the maximum of $S_N(M)-S_e(M)$ on $[2^k,2^{k+1}[$ is attained at $2^k+2^{k-1}-1$ and is given by 
  $$S_e(2^{k+1}-1) - 2 S_e(2^k+2^{k-1}-1) + S_e(2^k-1)$$
  because
  \begin{eqnarray*}
    S_N(2^k+2^{k-1}-1)&=&S_N(2^k-1) +\sum_{j=2^k}^{2^k+2^{k-1}-1}N(j)\\
  &=&S_e(2^k-1)+\sum_{j=2^k+2^{k-1}}^{2^{k+1}-1}e(j)\\
  &=&S_e(2^k-1)+S_e(2^{k+1}-1) -  S_e(2^k+2^{k-1}-1).\\    
  \end{eqnarray*}
For $M\in [2^k,2^{k+1}[$, we also get
\begin{equation}
  \label{eq:SNSe}
  S_N(M)-S_e(M)=\sum_{j=2^k}^M (N(j)-e(j)).
\end{equation}
  \begin{lemma}\label{lem:tec}
    For all $j\ge 1$, 
  $$N(j)-e(j)=e(\alpha(j))-e(j)\in 2(\alpha(j)-j) +\{-1,0,1\}$$
  and moreover, for two consecutive indices, 
  $$\sum_{r=0}^1 \left[N(2j+r)-e(2j+r)-2\left(\alpha(2j+r)-(2j+r)\right)\right]=0.$$  
\end{lemma}
\begin{proof}
The first part is obvious since we have $e(\alpha(j))\in 2\alpha(j)+\{0,1\}$ and $e(j)\in 2 j+\{0,1\}$ for any $j$. 

For the second part, let $u$ be the base-$2$ expansion of $j$. The four terms in 
  $$2(2j+2j+1)-e(2j)-e(2j+1)$$
  are respectively represented by $u00$, $u10$, $u0a$ and $u1(1-a)$ for some $a\in\{0,1\}$ such that $u0a$ and $u1(1-a)$ have an even number of ones. So, this sum is equal to $-1$. By definition of $N$ and $\alpha$, observe that the remaining terms can be grouped as 
  $$N(2j)-2\alpha(2j)=0\text{ and }N(2j+1)-2\alpha(2j+1)=1$$
  and the conclusion follows.
\end{proof}

  As a consequence of this lemma, for $M\in [2^k,2^{k+1}[$, we get
  $$S_N(M)=S_e(M)+2\sum_{j=2^k}^M \alpha(j) 
  -\left(M-2^k+1\right) \left(2^k+M\right) +R$$
  with $R\in\{-1,0,1\}$.
  If $M$ is odd, then the number of terms in the sum \eqref{eq:SNSe} is even so by the second part of Lemma~\ref{lem:tec}, we get $R=0$. If $M$ is even, then only the first part of the lemma can be applied and replacing $N(j)-e(j)$ with $2(\alpha(j)-j)$ could lead to an offset of $\pm1$.
  
Let $\ell\ge 0$ such that $M=2^k+\ell$. Observe that the above sum has $\ell+1$ terms. If we group together every two consecutive terms, we can make use of \eqref{eq:alpha2} to get 
\begin{align*}
\alpha(4n)+\alpha(4n+1)=4\alpha(2n)+1 
\quad \text{and} \quad 
\alpha(4n+2)+\alpha(4n+3)=4\alpha(2n+1)+1
\end{align*}
so 
\begin{equation}\label{eq: iterate alpha}
\sum_{j=2^k}^{M} \alpha(j) = 4 \sum_{j=2^{k-1}}^{2^{k-1}+\lfloor \frac{\ell-1}{2}\rfloor} \alpha(j) + \left\lfloor \frac{\ell-1}{2}\right\rfloor +1 + [(\ell+1) \% {2}]\, \alpha(M),
\end{equation}
where the last term only appears when $\ell$ is even since, in that case, the sum has an odd number of terms and the last term has thus to be treated separately.
By using~\eqref{eq: iterate alpha} repeatedly, one can write $S_N(M)$ as the sum of $S_e(M)-\left(M-2^k+1\right) \left(2^k+M\right) +R$ and $k$ terms of the form $\left\lfloor \frac{\ell'-1}{2}\right\rfloor +1 + [(\ell'+1) \% {2}]\, \alpha(2^{k'}+\ell')$ for decreasing values of $k',\ell'$, each term being multiplied by $2\cdot 4^{k-k'}$.

\begin{remark}
Since the main term in $S_e(M)$ is quadratic (recall the formula obtained in \cite{All}), in Figure~\ref{fig:parabola} we compare, on some interval $[2^k,2^{k+1}[$, $S_N(M)-S_e(M)$ and the parabola
$-2 M^2+6\cdot 2^k M- 4 \cdot 2^{2k}$ (which can be obtained from the intersections with the axis $y=0$ and knowing the maximum of the function).
\begin{figure}[h!tb]
  \centering
      \includegraphics[height=4.5cm]{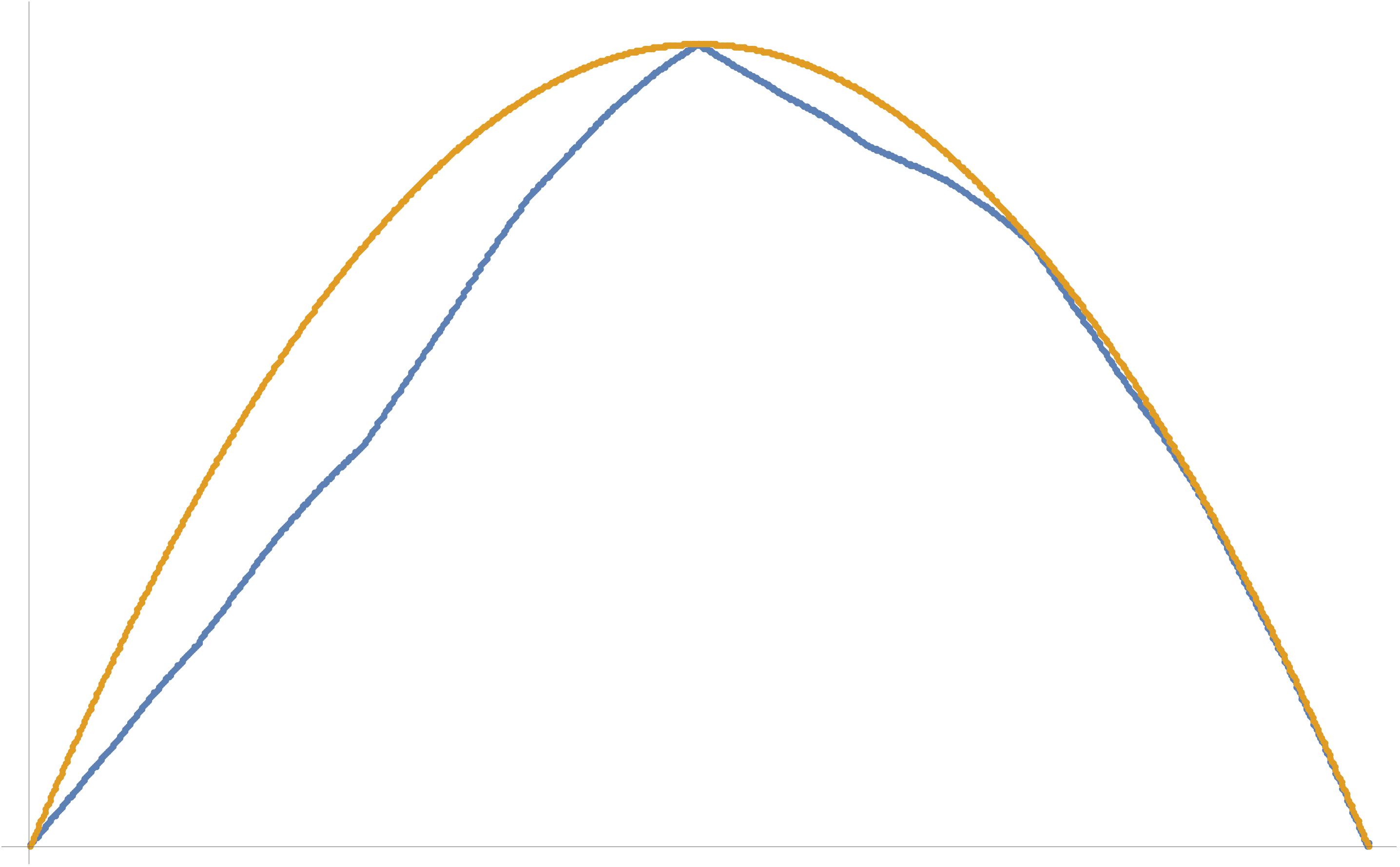}
  \caption{Comparison of $S_N(m)-S_e(M)$ with a parabola between two consecutive powers of $2$.}
  \label{fig:parabola}
\end{figure}
\end{remark}

   \section{Extension to trinomial coefficients}\label{sec:trinomial}
 One can also consider the generalization of Pascal's triangle to a three-dimensional pyramid made of trinomial coefficients (see, for instance, \cite{Wolfram}). Let $n\ge 0$. The plane of equation $x+y+z=n$ with $x,y,z\ge 0$ contains $(n+1)(n+2)/2$ integer points with value
 $$\binom{n}{x,y,z}=\frac{n!}{x!\, y!\, z!}.$$
 If these trinomial coefficients depicted by unit cubes are colored with respect to their value modulo~$p$, we get representations like the one in Figure~\ref{fig:pyr15}. In this section, we will generalize the observations from Section~\ref{sec:rec} and the recursive formula~\eqref{eq:z2n}.
 \begin{figure}[h!tb]
   \centering
   \includegraphics[height=6cm]{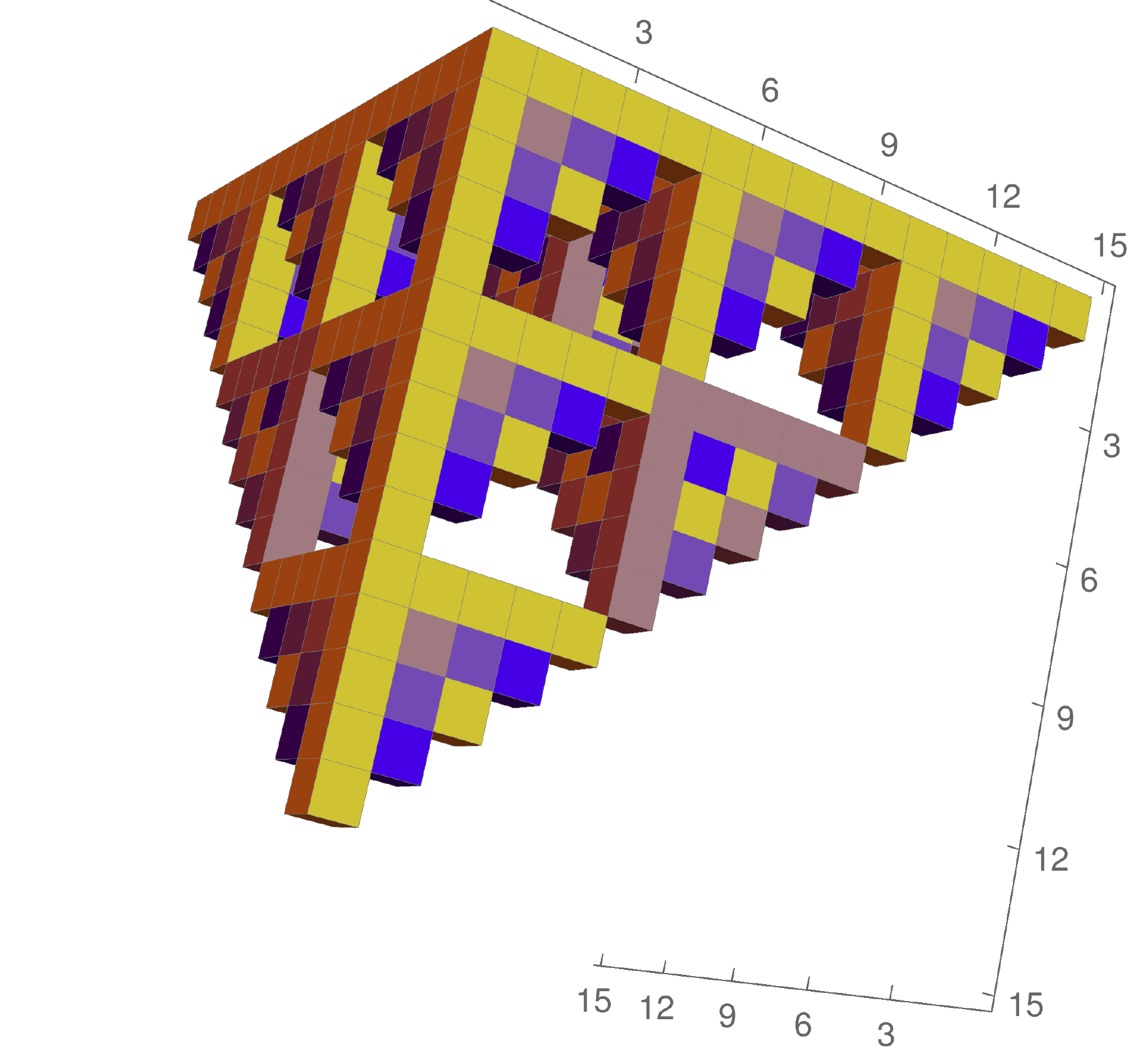}\quad
   \includegraphics[height=6cm]{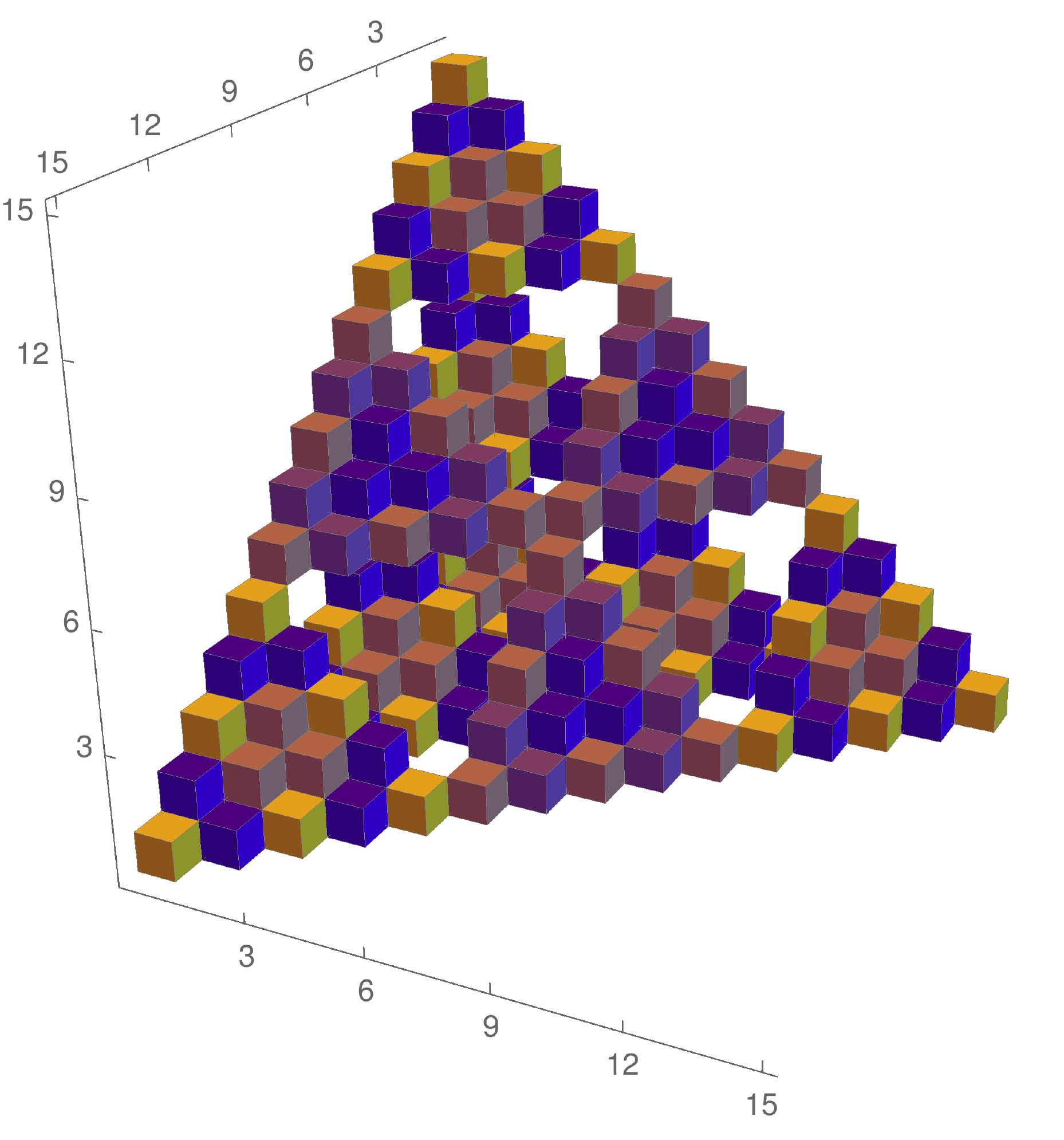}
   \caption{The first levels of Pascal's pyramid modulo $5$.}
   \label{fig:pyr15}
 \end{figure}
 
 For instance, for $n=5$, the sixth plane $x+y+z=5$ of the pyramid is a triangle that contains six rows ordered for $y=5,4,\ldots,0$. Since the coefficients are symmetric in the variables, one can also let vary either $x$ or $z$ (and take instead columns or diagonals of the form $x+y=z$). In the subsequent figures, we assume as usual that the $x$-axis is horizontal and the $y$-axis is vertical.

 \begin{definition}
Let $0\le k\le n$. We take these trinomial coefficients modulo $p$, so the \emph{$k$th line} (i.e., $z=n-k$) in the $n$th plane (i.e., $x+y+z=n$) of the pyramid is the base-$p$ expansion of an integer $\mathbf{t}_{p,n,k}$ defined by 
$$\mathbf{t}_{p,n,k}=\sum_{i=0}^k\left[\binom{n}{i,k-i,n-k}\bmod{p}\right]\, p^i.$$
\end{definition}

For $p=2$, if we order the elements plane by plane, and then for each plane, by row of increasing length, we get a sequence $\mathbf{t}_{2,0,0},\mathbf{t}_{2,1,0},\mathbf{t}_{2,1,1},\mathbf{t}_{2,2,0},\mathbf{t}_{2,2,1},\mathbf{t}_{2,2,2},\mathbf{t}_{2,3,0},\ldots$ whose first few terms are
$$1| 1, 3| 1, 0, 5| 1, 3, 5, 15| 1, 0, 0, 0, 17| 1, 3, 0, 0, 17, 51|\cdots.$$
See, for instance, Figure~\ref{fig:pyr6} for $n=5$. 
Note that $\mathbf{t}_{2,n,n}=\mathbf{t}_{2,n}$ for all $n\ge 0$, because the boundaries of the $n$th plane are copies of the $n$th row of Pascal's triangle.
 \begin{figure}[h!]
  $$
 \begin{array}{cccccc}
 &&&&&1\\
 &&&&5& 5\\
 &&&10& 20& 10\\
 &&10& 30& 30& 10\\
 &5& 20& 30& 20& 5\\
 1& 5& 10& 10& 5& 1\\
 \end{array}\quad 
  \begin{array}{|cccccc|l}
 &&&&&1  &  \mathbf{t}_{2,5,0}=1\\
 &&&&1& 1&  \mathbf{t}_{2,5,1}=3\\
 &&&0& 0& 0& \mathbf{t}_{2,5,2}=0\\
 &&0& 0& 0& 0& \mathbf{t}_{2,5,3}=0\\
 &1& 0& 0& 0& 1& \mathbf{t}_{2,5,4}=17\\
 1& 1& 0& 0& 1& 1& \mathbf{t}_{2,5,5}=51\\
 \end{array}
 $$
 \caption{The sixth row $(\mathbf{t}_{2,5,k})_{0\le k \le 5}$ of Pascal's pyramid.}
   \label{fig:pyr6}
 \end{figure}

Expand $(a+b+c)^n$ by the multinomial theorem and consider the coefficient of $a^ib^jc^k$ with $i+j+k=n$. It is equal to the corresponding coefficient in the product $(a+b+c)\cdot (a+b+c)^{n-1}$ where the latter factor is again expanded by the multinomial theorem. We get a generalization of Pascal's rule
$$\binom{n}{i,j,k}=\binom{n-1}{i-1,j,k}+\binom{n-1}{i,j-1,k}+\binom{n-1}{i,j,k-1}.$$

From this tree-term relation, we get the generalization of \eqref{eq:z2n} 
$$\forall i,j: 0\le j\le i,\quad \mathbf{t}_{2,i,j}=\mathbf{t}_{2,i-1,j}\oplus \mathbf{t}_{2,i-1,j-1}\oplus 2 \mathbf{t}_{2,i-1,j-1}$$
where we assume that $\mathbf{t}_{2,i,j}=0$ whenever $j>i$ or $j<0$.

\subsection{Lucas' theorem again and again} As in Section~\ref{sec:rec}, let us compare the values modulo~$p$ taken in the $i$th plane $x+y+z=i$ with $i=d\cdot p^k+s$, $0<d<p$ and $0\le s<p^k$, and those in the $s$th plane $x+y+z=s$. We will explain that the pattern modulo~$p$ of the $s$th plane repeats itself $(d+1)(d+2)/2$ times as square patches of size $p^k\times p^k$ under some well-understood permutation. 
In Figure~\ref{fig:plane323}, we consider $p=5$, $i=23$, $d=4$, $k=1$, and $s=3$.

\begin{figure}[h!tb]
  \centering
   \includegraphics[height=3.5cm]{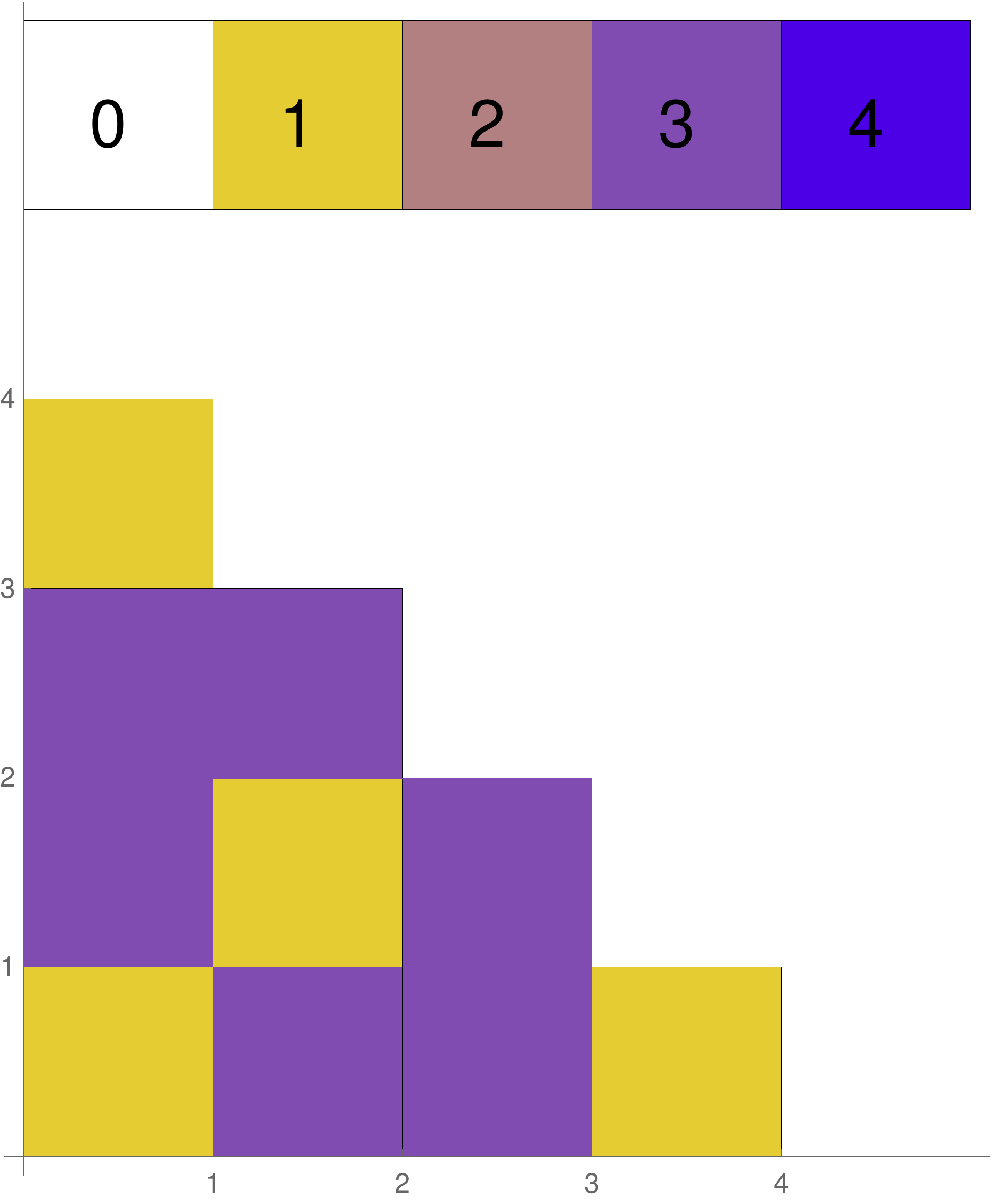}\quad\quad 
   \includegraphics[height=7cm]{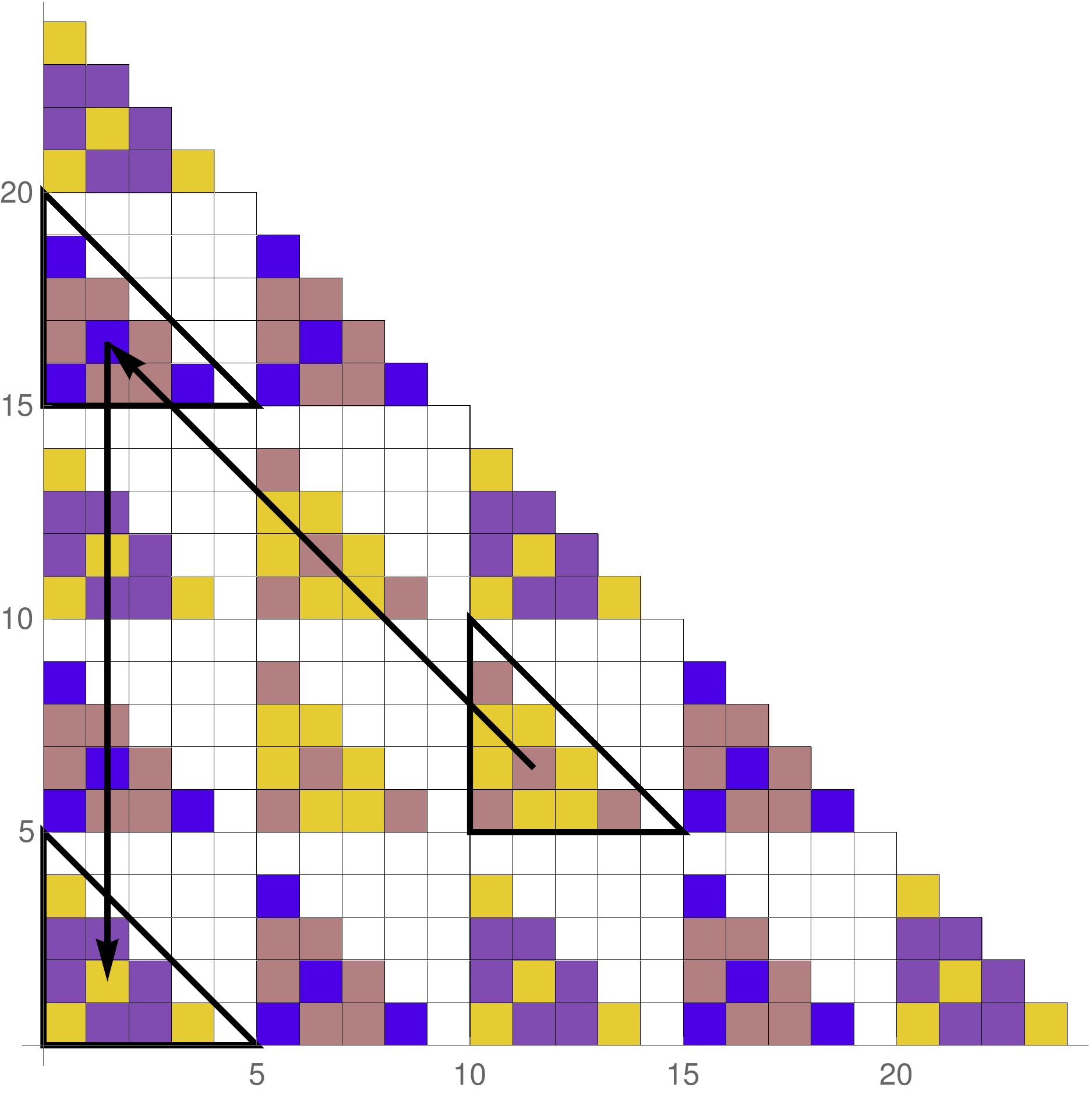}
  \caption{The $3$rd and $23$rd planes of Pascal's pyramid modulo~$5$. }
  \label{fig:plane323}
\end{figure}

Let $i=x+y+z$. We let $x_k\ge 0$ denote the quotient of $x$ by $p^k$. Let us compare the values modulo~$p$ of 
$$\binom{i}{x,\ y,\ z} \text{ with }\binom{i}{x-x_kp^k,\ y+x_kp^k,\ z}.$$
Geometrically, the map $(x,y,z)\mapsto (x-x_kp^k,\ y+x_kp^k,\ z)$ corresponds to a translation parallel to a side of the triangular boundary of the plane. For instance, in Figure~\ref{fig:plane323} where $p=5$, $i=23$ and $k=1$, if we take $x=11$, then $x_k=2$ and we have depicted the corresponding translation vector. Adding $(-10,10,0)$ to $(x,y,z)$ does not change the sum of the three components (we remain in the plane $x+y+z=23$) but translates the $5\times 5$ square region bounded by $10\le x<15$ and $5\le y<10$ to the square $0\le x'<5$ and $15\le y'<20$. Similarly, we could have considered a transformation of the form $(x,y,z)\mapsto (x-x_kp^k,\ y,\ z+x_kp^k)$. Due to the symmetry of the trinomial coefficients, six such transformations can be considered and correspond to translations in two directions parallel to one of the three sides of the boundary. We will indeed compose two such translations. On the one hand, recalling that $i=x+y+z$, we have
\begin{eqnarray*}
\binom{i}{x-x_kp^k,\ y+x_kp^k,\ z}&=&
\frac{i!}{z!\, (x+y)!} \cdot \frac{(x+y)!}{(x-x_kp^k)!\, (y+x_kp^k)!}\\
                                  &=&\binom{i}{z}\, \binom{i-z}{x-x_kp^k}\\
  &\equiv& \binom{i}{z}\ \underbrace{\binom{\epsilon_k(i-z)}{0}}_{=1} \ \prod_{j=0}^{k-1} \binom{\epsilon_j(i-z)}{\epsilon_j(x)} \bmod{p},\\
\end{eqnarray*}
where $\epsilon_j(n)$ is the $j$th least significant digit in the base-$p$ expansion of $n$ (and leading zeroes are allowed, for instance, $\epsilon_k(i-z)=0$ whenever $i-z<p^k$). On the other hand, we find
\begin{eqnarray*}
\binom{i}{x,\ y,\ z} &=&\binom{i}{z}\, \binom{i-z}{x}\\
  &\equiv& \binom{i}{z} \binom{\epsilon_k(i-z)}{x_k} \prod_{j=0}^{k-1} \binom{\epsilon_j(i-z)}{\epsilon_j(x)}\quad \pmod{p}.\\
\end{eqnarray*}
Hence with the same notation as in Section~\ref{sec:rec}, we have
$$\mu_{\epsilon_k(i-z),x_k}^{-1} \binom{i}{x,\ y,\ z}\equiv  \binom{i}{x-x_kp^k,\ y+x_kp^k,\ z}\pmod{p}.$$
We now apply a second map of the form $(x',y',z')\mapsto (x',y'-y'_k,z'+z'_k)$ and we get the permutation $\mu_{\epsilon_k(i-x'),y'_k}^{-1}$ acting of the values of the coefficients of the translated region. Combining these two transformations, we may relate the value modulo~$p$ of the initially considered trinomial coefficient with the trinomial coefficient of some
$$\binom{i}{x',\ y',\ z'} \text{ with } x',y'<p^k.$$
In Figure~\ref{fig:plane323}, adding $(0,-15,15)$ to $(x,y,z)$ does not change the sum of the three components but translate the $5\times 5$ square region bounded by $0\le x<5$ and $15\le y<20$ to the square $0\le x'<5$ and $0\le y'<5$. Consequently, the values modulo~$p$ are modified according to the composition of two permutations of the form $\mu_{a,b}$.

To conclude with the example given in Figure~\ref{fig:plane323}, start from the region $10\le x<15$, $5\le y<10$. First consider the sub-region with the extra constraint $15\le x+y\le 18$, we make such a splitting to consider the ``colored'' region and avoid ambiguity about $\epsilon_1(x+y)$. So we have $\epsilon_1(x+y)=3$ and $\epsilon_1(x)=x_1=2$. We thus consider a multiplication by the inverse (modulo~$5$) of the coefficient $\binom{3}{2}=3$ which is $2$ --- the reader may compare the two colored triangles connected by a diagonal arrow: they are off by a multiple of $2$. Now inside the region $0\le x<5$, $15\le y<20$, $15\le x+y\le 18$, we observe that $x\le 3$. So $20 \le 23-x=y+z\le 23$. We thus have $\epsilon_1(y+z)=4$ and $\epsilon_1(y)=y_1=3$. So we consider a multiplication by the inverse (modulo~$5$) of the coefficient $\binom{4}{3}=4$ which is $4$
--- the reader may compare the two colored triangles connected by a vertical arrow: again they are off by a multiple of $4$.

We have not discussed yet the white region corresponding to coefficients congruent to zero. For the region $10\le x<15$, $5\le y<10$, $19\le x+y<25$, so we first have a multiplication by the inverse of $\binom{\epsilon_1(x+y)}{2}$. But such a computation is irrelevant, because $\mu_{a,b}(0)=0$ for all $a,b$ ; meaning that white squares are mapped to white squares for all the considered translations. Another way to see this phenomenon is explained in the next subsection. 

\subsection{A $p$-automatic pyramid}

Let us quote Granville about Pascal's triangle modulo~$p$: ``\emph{Lucas' theorem may be viewed as a result about automata with $p$ possible states!}'' \cite{Granville}. Let us also mention \cite{AllBer} where a substitution mapping elements from $\{0,\ldots,p-1\}$ to $(p\times p)$-blocks allows the authors to  compute the rectangular block complexity of the associated bidimensional sequence. For $p=2$, the iterated substitution is 
\begin{equation}
  \label{eq:substi}
  1\mapsto\begin{array}{|c|c|}
            \hline
            1&0\\
            \hline
            1&1\\
            \hline
            \end{array}\, ,\ 0\mapsto\begin{array}{|c|c|}
            \hline
            0&0\\
            \hline
            0&0\\
            \hline
            \end{array}\, .
\end{equation}
With a reasoning similar to the one of the previous subsection, we show that such a construction still holds in higher dimension. Note that for dimension $2$ (thus from Pascal's triangle), it is known that $\left(\binom{m}{n}\bmod{d}\right)_{m,n\ge 0}$ is $k$-automatic for some integer $k\ge 0$ if and only if $d$ is a power of a prime $p$. In that case, the sequence is $p$-automatic \cite{AH}. Here we show that $\left(\binom{x+y+z}{x,\, y,\, z}\bmod{p}\right)_{x,y,z\ge 0}$ is $p$-automatic.

\begin{lemma}
If $0\le x\le y<p$ and $x+y\ge p$, then $x>\lfloor (x+y)/p\rfloor$.  
\end{lemma}
\begin{proof}
Since $x+y<2p$ and $x+y\ge p$, write $x+y=p+r$ with $r=\lfloor (x+y)/p\rfloor<p$. Thus we get $x=r+p-y$ with $p-y>0$.  
\end{proof}

The next result permits us to obtain the values modulo~$p$ of trinomial coefficients within the cube $p (x,y,z)+\{0,\ldots,p-1\}^3$ from the value at $(x,y,z)$. The other way round, a value at a specific position is determining $p^3$ values at further positions in the space. In particular, the statement also explains why a pyramid is created. The cube is cut by a plane $x+y+z=p$ and is thus split into two regions. In the subset belonging to the half-space $x+y+z\ge p$, values modulo~$p$ of the coefficients are zero.

\begin{proposition}\label{pro:tri}
Let $a,b,c\in\{0,\ldots,p-1\}$.  If $a+b+c<p$, then 
$$  \binom{p(x+y+z)+a+b+c}{px+a,\  py+b,\  pz+c} \equiv
\binom{a+b+c}{a} \binom{b+c}{b} \binom{x+y+z}{x,\ y,\ z}\pmod{p}.$$
Otherwise
$$  \binom{p(x+y+z)+a+b+c}{px+a,\  py+b,\  pz+c} \equiv 0\pmod{p}.$$
\end{proposition}
\begin{proof}
Without loss of generality, assume $a\le b\le c$. We have
$$\binom{p(x+y+z)+a+b+c}{px+a,\  py+b,\  pz+c} = \binom{p(x+y+z)+a+b+c}{px+a}\binom{p(y+z)+b+c}{py+b}.$$
Let $u_k\cdots u_0$, $v_k\cdots v_0$, $x_k\cdots x_0$, $y_k\cdots y_0$ respectively be the base-$p$ expansions of $x+y+z$, $y+z$, $x$, $y$ such that $u_k\neq 0$. As usual, we allow leading zeroes to get expansions of the same length if necessary.
We examine two cases.

Suppose first that $a+b+c<p$. Then \[
\rep_p(p(x+y+z)+a+b+c)=u_k\cdots u_0 (a+b+c)\]
because $a+b+c$ is a single digit. In particular, $b+c<p$ and $\rep_p(p(y+z)+b+c)=v_k\cdots v_0 (b+c)$.
We may apply Lucas' theorem to both binomial coefficients to obtain 
\begin{eqnarray*}
  \binom{p(x+y+z)+a+b+c}{px+a,\  py+b,\  pz+c} &\equiv& \prod_{i=0}^k \binom{u_i}{x_i} \binom{a+b+c}{a} \prod_{j=0}^k \binom{v_j}{y_j} \binom{b+c}{b} \pmod{p}\\
                                               &\equiv& \binom{a+b+c}{a} \binom{b+c}{b} \binom{x+y+z}{x,\ y,\ z}\pmod{p},
                                               \end{eqnarray*}
as expected.

Assume now that $a+b+c\ge p$. As a first sub-case, assume $b+c\ge p$. By the above lemma, $b>\lfloor (b+c)/p\rfloor$. If we compute the base-$p$ expansion of $p(y+z)+b+c$, the last digit is $\lfloor (b+c)/p\rfloor$ followed by $\rep_p(p(y+z)+1)=v_k'\cdots v_0'$; there is a carry to deal with. Applying as above Lucas' theorem yields
                                             $$\binom{p(x+y+z)+a+b+c}{px+a,\  py+b,\  pz+c} \equiv \prod_{i=0}^k \binom{u_i}{x_i} \binom{a+b+c}{a} \prod_{j=0}^k \binom{v_j'}{y_j} \underbrace{\binom{\lfloor (b+c)/p\rfloor}{b}}_{=0} \pmod{p},$$
                                             as desired.
                                             As a final sub-case, assume that $a\le b+c< p$. By the above lemma, $a>\lfloor (a+(b+c))/p\rfloor$. Now the reasoning is similar. If we compute the base-$p$ expansion of $p(x+y+z)+a+b+c$, the last digit is $\lfloor (a+b+c)/p\rfloor$ followed by $\rep_p(p(x+y+z)+1)=u_k'\cdots u_0'$. Applying as above Lucas' theorem yields
                                             $$\binom{p(x+y+z)+a+b+c}{px+a,\  py+b,\  pz+c} \equiv \prod_{i=0}^k \binom{u_i}{x_i} \underbrace{\binom{\lfloor (a+b+c)/p\rfloor}{a}}_{=0} \prod_{j=0}^k \binom{v_j'}{y_j} \binom{b+c}{b} \pmod{p},$$
                                             as wanted.
\end{proof}

This proposition permits us to define a 3D-substitution over $\{0,\ldots,p-1\}^3$ similar to \eqref{eq:substi} or, equivalently, an automaton reading triplets of digits. The initial symbol is $1$. The image of a symbol $q\in\{0,\ldots,p-1\}$ is a cube of size $p$ indexed by $\{0,\ldots,p-1\}^3$ such that, for all  $a,b,c\in\{0,\ldots,p-1\}$, 
if $a+b+c\ge p$, then $[\sigma(q)]_{a,b,c}=0$ and  if $a+b+c<p$, then  
$$\sigma(q)_{a,b,c}=q\binom{a+b+c}{a} \binom{b+c}{b} \mod{p}.$$
See Figure~\ref{fig:sub} for an example of images of $\sigma$ in the case $p=5$.
Observe that iterations of $\sigma$ on $1$ are converging. Indeed, if we iterate $\sigma$ on $1$, then $\sigma^n(1)$ is a cube of size $p^n$ and $\sigma^n(1)$ appears inside $\sigma^{n+1}(1)$ at the origin $(0,0,0)$.
  \begin{figure}[h!tb]
  \centering
   \includegraphics[height=4.5cm]{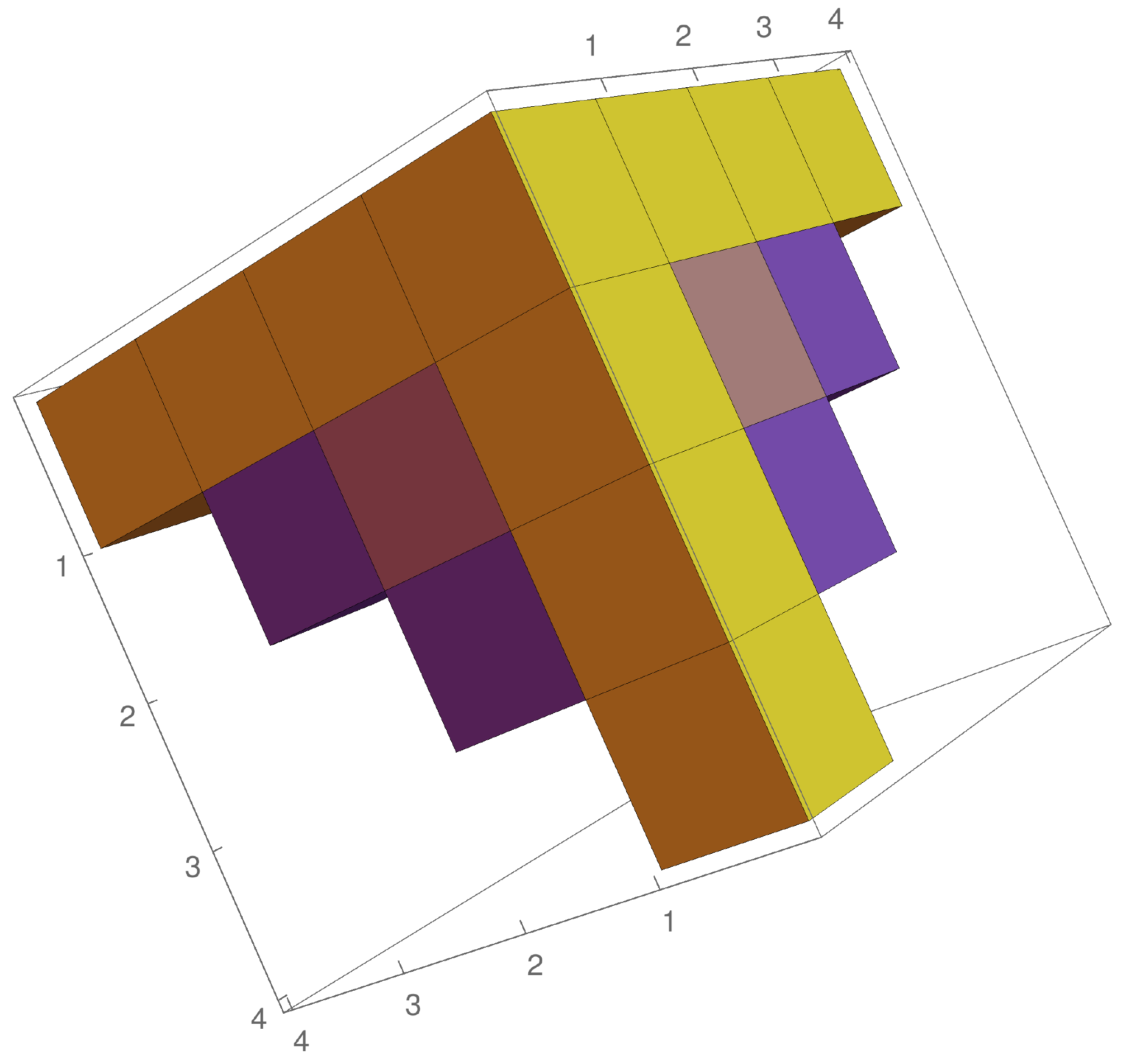}\quad
   \includegraphics[height=4.5cm]{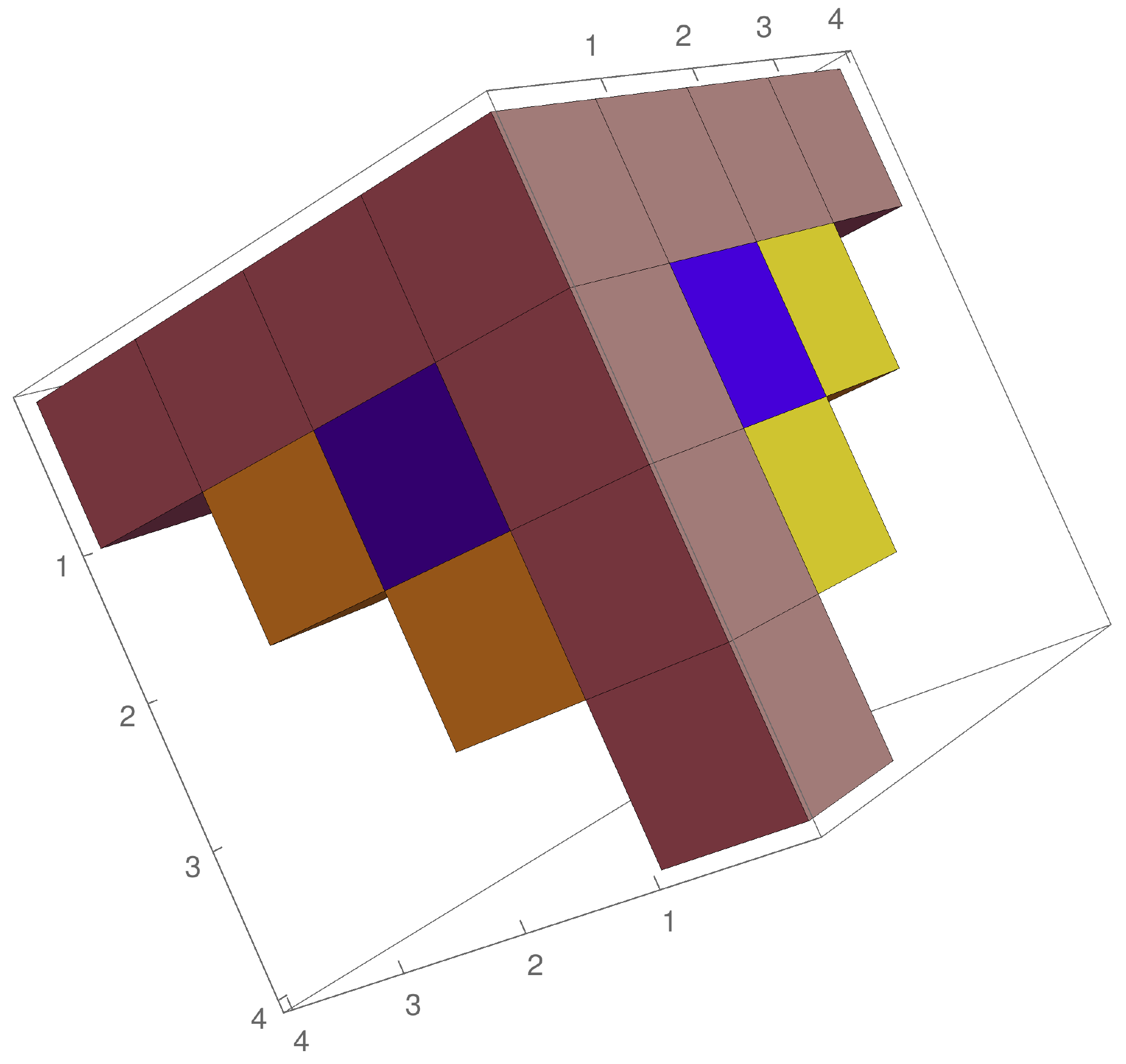}\\
      \includegraphics[height=4.5cm]{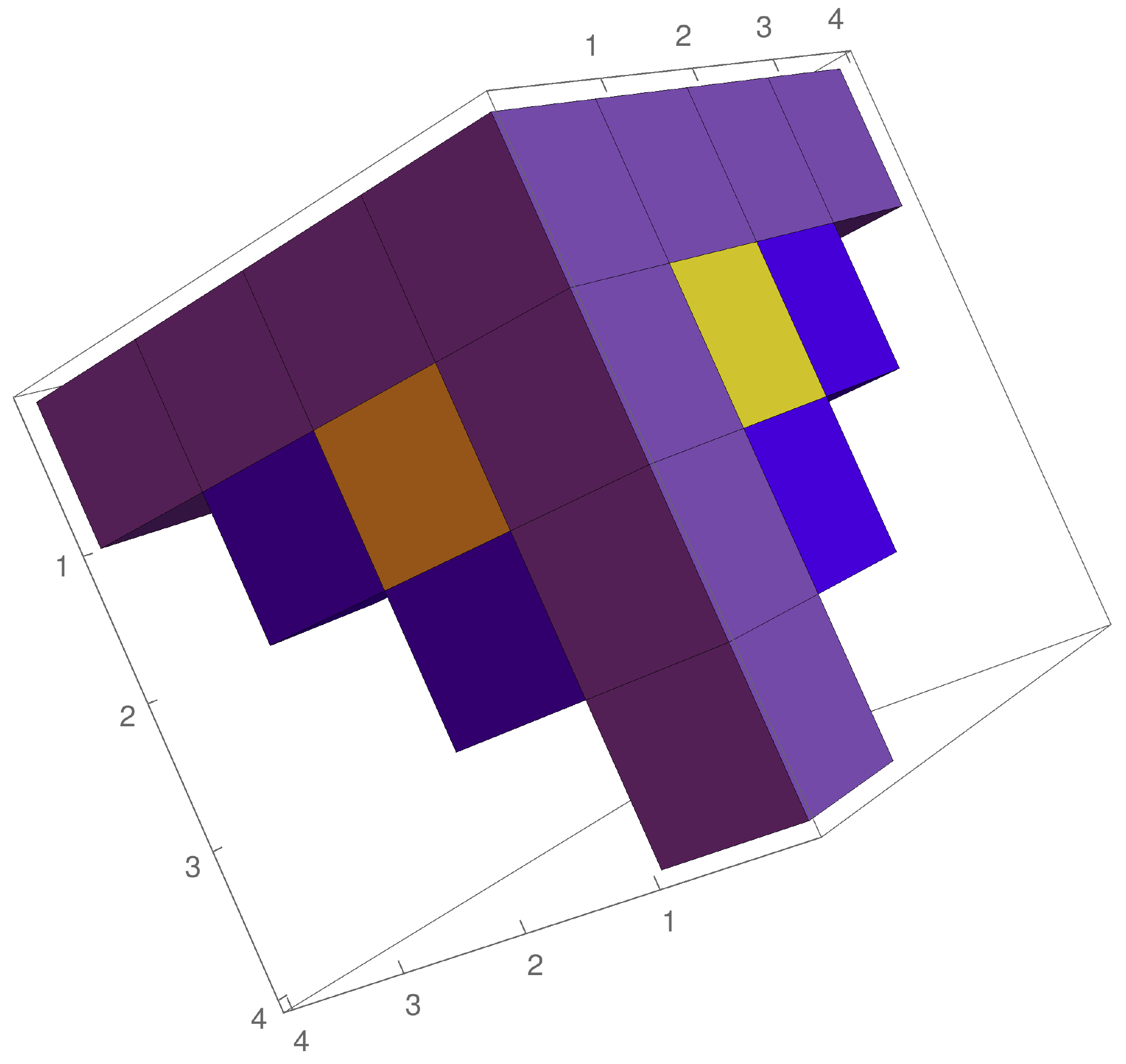}\quad
   \includegraphics[height=4.5cm]{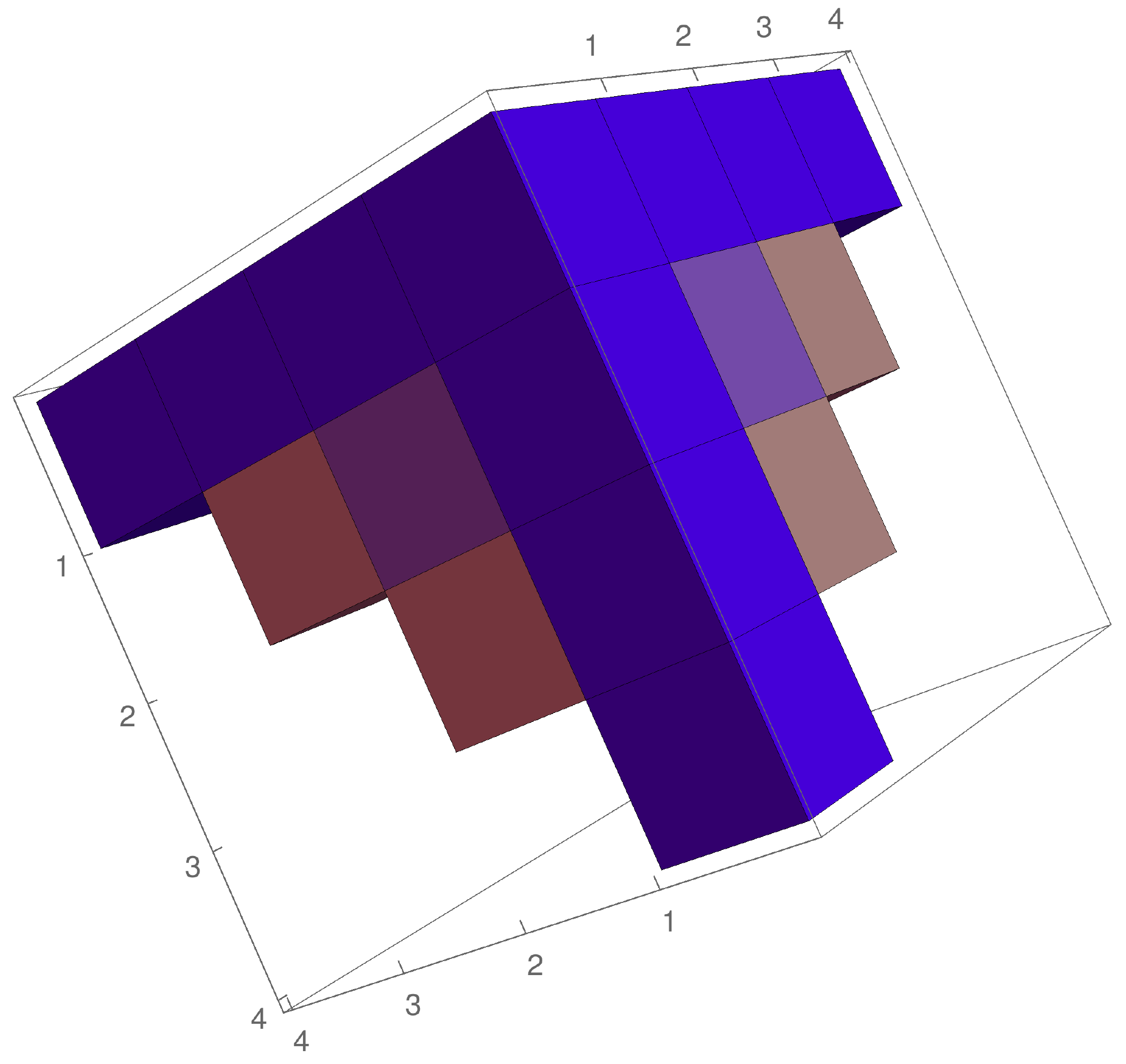}
  \caption{The images $\sigma(q)$ for $p=5$ and $q\in\{1,2,3,4\}$ (in this order).}
  \label{fig:sub}
\end{figure}
\section{Concluding remarks}\label{sec:conclusion}

In Section~\ref{sec:nim}, we focused on base $2$. For a general integer base $p>2$, with $\oplus_p$ being the addition digit-wise modulo~$p$ (without carry), it is obvious that $\mathbf{t}_{p,n+1}=\mathbf{t}_{p,n}\oplus_p (p\, \mathbf{t}_{p,n})$. We can introduce a sequence $N_p(m)$ defined by $(m\oplus_p p.m)_{m\ge 0}$. Nevertheless, except for $p=3$ with {\tt A242399}, no such sequences appear in the OEIS and contrarily to the binary case, we do not find any nice property to report.

In \cite{LRS}, we have considered generalizations of Pascal's triangle to binomial coefficients of words. When these coefficients are reduced modulo $p$, we could also define an analogue of the sequence $(\mathbf{t}_{p,n})_{n\ge 0}$. A natural candidate to consider is the Fibonacci numeration system, i.e., the words of the numeration language belong to $1\{0,01\}^*\cup\{\varepsilon\}$. The rows of this Pascal's triangle modulo~$2$ evaluated as base-$2$ expansions give the sequence whose first terms are
$$1, 3, 5, 5, 29, 9, 57, 129, 249, 177, 705, 3681,\ldots$$
and evaluating these rows as Fibonacci representations (not necessarily greedy) gives
$$1, 3, 4, 4, 17, 6, 27, 35, 82, 56, 145, 501, 624, 22, 1056,\ldots.$$

In the last section, we considered trinomial coefficients but the reasoning can be extended to multinomial coefficients. In particular, Proposition~\ref{pro:tri} can be extended showing that the multidimensional sequence $\left(\binom{x_1+\cdots +x_n}{x_1,\, \ldots,\, x_n}\bmod{p}\right)_{x_1,\ldots,x_n\ge 0}$ is $p$-automatic.

With Proposition~\ref{pro:tri}, one could also think about a possible connection with the so-called {\em combinatorial numeration system} where every integer can be decomposed as a sum of binomial coefficients of a prescribed form \cite{CRS,Katona}.

\section*{Acknowledgment}
Manon Stipulanti is supported by the FNRS Research grant 1.B.397.20.

\end{document}